\documentclass{amsart}
\usepackage[all]{xy}
\usepackage{graphicx}
\usepackage{amssymb, amsmath,  amsfonts,  color}

\newtheorem{thm}{Theorem}[section]
\newtheorem{prop}[thm]{Proposition}
\newtheorem{lemma}[thm]{Lemma}

\newtheorem{cor}[thm]{Corollary}
\newtheorem{defn}[thm]{Definition}
\newtheorem{example}[thm]{Example}
\newtheorem{remark}[thm]{Remark}

 \numberwithin{equation}{section}

\begin{document}
\title[On Seidel representation in $QK(Gr(k,n))$]{ On  Seidel representation in quantum $K$-theory of   Grassmannians}


\author{Changzheng Li}
 \address{School of Mathematics, Sun Yat-sen University, Guangzhou 510275, P.R. China}
\email{lichangzh@mail.sysu.edu.cn}

\author{Zhaoyang Liu}
\address{School of Mathematics, Sun Yat-sen University, Guangzhou 510275, P.R. China}
\email{liuchy73@mail2.sysu.edu.cn}

\author{Jiayu Song}
 \address{School of Mathematics, Sun Yat-sen University, Guangzhou 510275, P.R. China}
\email{songjy29@mail2.sysu.edu.cn}

\author{Mingzhi Yang}
 \address{School of Mathematics, Sun Yat-sen University, Guangzhou 510275, P.R. China}
\email{yangmzh8@mail2.sysu.edu.cn}


\thanks{}
\date{
      }

 \dedicatory{Dedicated to the memory of Bumsig Kim}



\begin{abstract}
 We provide a direct proof of Seidel representation  in the quantum $K$-theory  $QK(Gr(k, n))$ by studying projected Gromov-Witten varieties concretely. As applications, we give an alternative proof of
  the $K$-theoretic quantum Pieri rule by Buch and Mihalcea,
     reduce  certain quantum Schubert structure constants of higher degree to  classical Littlewood-Richardson coefficients for $K(Gr(k, n))$,
    and
    provide a quantum Littlewood-Richardson rule for $QK(Gr(3, n))$.
   \end{abstract}

\maketitle

 \tableofcontents

\section{Introduction}

 Let $\mbox{Ham}(M, \omega)$ denote the group of Hamiltonian  symplectomorphisms on a  symplectic manifold  $(M, \omega)$.
 Seidel  representation is an action of the fundamental group $\pi_1(\mbox{Ham}(M, \omega))$ on the invertible quantum cohomology  $QH^*(M)^{\times}$. It was first constructed by Seidel for monotone symplectic manifolds  \cite{Seid},
and was later extended to all symplectic manifolds \cite{McDu, McTo}.  While Seidel  representation    has very nice applications for toric manifolds (see  \cite{GoIr, CLLT} and references therein),   one of its first two applications 
was   for 
complex Grassmannians $X=Gr(k, n)$ \cite{Seid}. In this case, the $U(n)$-action on $\mathbb{C}^n$ induces a Hamiltonian action of projective unitary group $PU(n)$ on $X$. In return, Seidel constructed an action of the cyclic group  $\pi_1(PU(n))=\mathbb{Z}/n\mathbb{Z}$ on the quantum cohomology $QH^*(X)$ (with specialization  at $q=1$). There was an explanation of this cyclic symmetry from the viewpoint of Verlinde algebra  by Agnihotri and   Woodward \cite[Proposition 7.2]{AgWo}. The Grassmannian $X$ is a special case of a flag variety $G/P$, where $G$ is a simple simply connected complex Lie group and $P$ is a parabolic subgroup of $G$. In \cite{Belk}, Belkale constructed an injective group homomorphism from the center    of $G$ to the Weyl group of $G$, and 
 obtained a transformation formula for genus zero, three-pointed Gromov-Witten invariants for $G/P$ in a geometric way. This   recovers the $(\mathbb{Z}/n\mathbb{Z})$-symmetry for $QH^*(X)$ again. There was also   a combinatorial aspect of this cyclic symmetry by Postnikov \cite[Proposition 6.1]{Post}. The Seidel representation in   the quantum cohomology of  complete flag  manifolds in Lie type $A$ was constructed by Postnikov in \cite{Post2001} and for cominuscule Grassmannians by Chaput, Manivel and Perrin in  \cite{CMP}.

  The small quantum cohomology $QH^*(X)=(H^*(X,\mathbb{Z})\otimes \mathbb{Z}[q], \star)$ is a deformation of the classical cohomology $H^*(X,\mathbb{Z})$, and has a  $\mathbb{Z}[q]$-additive basis of Schubert classes $[X^\lambda]$ indexed by partitions $\lambda$ in $k\times (n-k)$ rectangle. Here $X^\lambda$ denotes a Schubert variety of codimension $|\lambda|$. The aforementioned cyclic symmetry is actually realized by an operator $T=[X^{(1,\cdots, 1)}]\star$ on $QH^*(X)$, where $X^{(1,\cdots, 1)}$ is a smooth Schubert  variety isomorphic to $Gr(k, n-1)$.
   The fact 
   $T^n=q^k\mbox{Id}$, together with the fact that $T^j$ is not a scalar map for $1\leq j<n$, shows that $T|_{q=1}$ generates a $(\mathbb{Z}/n\mathbb{Z})$-action on $QH^*(X)|_{q=1}$.

  Compared with  the extensive studies of $QH^*(X)$, much less is known about the quantum $K$-theory $QK(X)=(K(X)\otimes \mathbb{Z}[[q]], *)$. The classes $\mathcal{O}^\lambda=[\mathcal{O}_{X^\lambda}]$ of the structure sheaves over Schubert varieties $X^\lambda$ form an additive basis of $QK(X)$. The quantum $K$-product
     $$\mathcal{O}^\lambda*\mathcal{O}^\mu=\sum_{d\in\mathbb{Z}_{\geq 0}}\sum_\nu N_{\lambda, \mu}^{\nu, d}\mathcal{O}^\nu q^d$$
  is a priori a former power series in $q$ by definition, and turns out to be a polynomial in $q$ \cite{BuMi,BCMPfinite00,BCMPfinite, ACT}.
  The first detailed study of $QK(X)$ was due to Buch and Mihalcea in \cite{BuMi}, including  a Pieri rule calculating
     quantum $K$ products of the form  $\mathcal{O}^\lambda*\mathcal{O}^{(r, 0, \cdots, 0)}$.
The operator   $\mathcal{T}=\mathcal{O}^{(1, \cdots, 1)}*: QK(X)\to QK(X)$
could be viewed as a special case of their Pieri rule, thanks to the isomorphism between $X$ and its dual  Grassmannian $Gr(n-k, n)$. As one main result, we show the behavior of $\mathcal{T}$ directly, by studying relevant projected Gromov-Witten varieties carefully. It is a highlight of our alternative proof that we provide very concrete descriptions of the relevant curve neighborhoods in section 4.2 (especially \textbf{Proposition 4.5}).
This inspired a recent work \cite{Tari} of  Tarigradschi, which gives a  proof of the  conjecture on curve neighborhoods of Seidel products proposed in \cite{BCP23}.

\begin{thm}\label{SeidelQKT}  Let $\lambda\in \mathcal{P}_{k, n}=\{(\lambda_1, \cdots, \lambda_k)\in \mathbb{Z}^k\mid n-k\geq \lambda_1\geq \cdots \geq \lambda_k\geq 0\} $. In $QK(X)$, we have
  \begin{equation}\label{eqnTT}
     \mathcal{T}(\mathcal{O}^\lambda)= \mathcal{O}^{(1,\cdots, 1)}*\mathcal{O}^\lambda=\begin{cases}
     q \mathcal{O}^{(\lambda_2, \cdots, \lambda_{k},0)},&\mbox{if } \lambda_1=n-k,\\
      \mathcal{O}^{(\lambda_1+1,\cdots, \lambda_k+1)},&\mbox{if }\lambda_1<n-k.
  \end{cases}
  \end{equation}
 \end{thm}
 \noindent It follows that  $\mathcal{T}^n=q^{k}\mbox{Id}$, and  $\mathcal{T}|_{q=1}$ generates a $(\mathbb{Z}/n\mathbb{Z})$-action on $QK(X)|_{q=1}$, which we called Seidel representation.
     {As the first application, we use  Seidel representation   to reprove Buch and Mihalcea's  quantum Pieri rule in  \textbf{Proposition \ref{propQuanPieri22}}.} We remark that
        intersections of two Schubert varieties of $X$ associated to non-transverse complete flags arise in our alternative proof. To obtain the rational connectedness of such intersections, we make use of the refined double decomposition of intersections of Schubert cells of the complete flag variety by B. Shapiro, M. Shapiro and  Vainshtein \cite{SSV} (see also \cite{Curt}).
\begin{remark}
  It is worth to understand the quantum Pieri rule from the viewpoint of Seidel representation, which roughly tells us that
     {\upshape $$\mbox{quantum Pieri rule } = \mbox{classcial Pieri rule } + \mbox{Seidel action}.$$}

   {For quantum} cohomology of $Gr(k, n)$, this was shown by Belkale \cite{Belk}. Along this line,  the first and third named authors \cite{ LiSo} studied the quantum Pieri rule in the  quantum cohomology of complete flag variety $F\ell_n$, where
the special Schubert classes are those  pullback from   $Gr(1, n)$.
In   \cite{BCP23},  Buch, Chaput and  Perrin   studied quantum Pieri rule  for quantum $K$-theory of  cominuscule Grassmannians along this line as well.
 \end{remark}

   Let $\lambda\uparrow i$ be the unique partition satisfying  $\mathcal{T}^i(\mathcal{O}^\lambda)|_{q=1}=\mathcal{O}^{\lambda\uparrow i}$,
   called the $i$-th Seidel shift  of $\lambda$ \cite[Definition 15]{BuWa} (see Definition \ref{Seidelshifts} for precise descriptions).
   As the second  application of Seidel representation,
   we obtain the following.
  \begin{thm}\label{mainthm1}
     Let $\lambda, \mu \in \mathcal{P}_{k, n}$.
 The smallest power $d_{\rm min}$ of $q$ appearing in  $\mathcal{O}^\lambda*\mathcal{O}^\mu$ in $QK(X)$ equals that appearing in  $[X^\lambda]\star [X^\mu]$ in $QH^*(X)$, and   is given by 
$$d_{\rm min} = \max\{{1\over n}\big({|\lambda|-|\lambda\uparrow i|+|\mu|-|\mu\uparrow (n-i)|} \big)| 0\leq i\leq  n\}.$$ Moreover, if the max
is achieved for $r$, then
   \begin{equation}\label{Belprod}
      \mathcal{O}^\lambda* \mathcal{O}^\mu=q^{d_{\rm min}} \mathcal{O}^{\lambda\uparrow r}*\mathcal{O}^{\mu\uparrow (n-r)}.
   \end{equation}
  \end{thm}
\noindent In particular, the structure constants $N_{\lambda, \mu}^{\nu, d_{\min}}$ are all Littlewood-Richardson coefficients
   in $\mathcal{O}^{\lambda\uparrow r}\cdot\mathcal{O}^{\mu\uparrow (n-r)}$ in $K(X)$.
 \begin{remark}
       Belkale obtained the same formula as \eqref{Belprod}  in Theorem 10 of \cite{Belk} for the quantum cohomology $QH^*(Gr(k,n))$.   Postnikov also obtained  an equivalent
        formula by combining    Corollary 6.2 and the $D_{\rm min}$-part of Theorem 7.1 in \cite{Post}. The above theorem generalizes their results to the  quantum $K$-theory $QK(Gr(k, n))$.

        The first part on $d_{\rm min}$ has  also been proved independently in  \cite{BCMP22}. Therein  Buch, Chaput, Mihalcea and Perrin showed that the powers $q^d$ that occur in $\mathcal{O}^\lambda* \mathcal{O}^\mu$
           form an interval which is the same as that for the quantum product $[X^\lambda]\star[X^\mu]$ in the quantum cohomology of a (minuscule) Grassmannian. A formula for $d_{\rm min}$ in the quantum cohomology of a homogeneous variety $G/P$ was proved in \cite{FuWo}.

           Our theorem has a different emphasis in relating two quantum products, so that the coefficient of the smallest power $q^{d_{\rm min}}$ in a quantum product coincides  with a precisely described Littlewood-Richardson coefficient of another product in the $K$-theory
            of $Gr(k, n)$.
 \end{remark}

 Using the operators $\mathcal{T}, \mathcal{O}^{(n-k, 0,\cdots, 0)}$ and the duality in \cite[Theorem 5.13]{BuMi}, we also show  that certain structure constants $N_{\lambda, \mu}^{\nu, d}$ with $d\geq 1$ are equal to classical Littlewood-Richardson coefficients $N_{\tilde\lambda, \tilde \mu}^{\tilde \nu, 0}$. For instance, we provide an accessible   sufficient condition for  reduction for structure constants of degree one as follows.
\begin{thm}\label{thmdegreeone}
  Let $\lambda, \mu, \nu\in \mathcal{P}_{k, n}$ and $d\geq 1$. If   $\nu_i<\max\{\lambda_i, \mu_i\}$ for some $i$, then there exist $\tilde \lambda, \tilde \mu, \tilde \nu\in \mathcal{P}_{k, n}$ with explicit descriptions, such that  $N_{\lambda, \mu}^{\nu, d}=N_{\tilde \lambda, \tilde \mu}^{\tilde \nu, d-1}$.
\end{thm}
\noindent More precise  statement  of the above theorem will be provided  in \textbf{Theorem \ref{thmdegonerestated}}.
  We remark that  many reductions of quantum-to-classical type have been given by Postnikov \cite{Post} for quantum cohomology $QH^*(X)$, while the reduction for the largest power in the quantum product of two Schubert classes cannot be generalized to quantum $K$-theory due to the lack of strange duality.
  The structure constants $N_{\lambda, \mu}^{\nu, d}$ are not single  but combinations of $K$-theoretic Gromov-Witten invariants. This is completely different from that for $QH^*(X)$, making the study of  $N_{\lambda, \mu}^{\nu, d}$ much more complicated. The quantum-to-classical principle derived in \cite{BuMi} can only reduce  $K$-theoretic Gromov-Witten invariants of $X$ to classical $K$-intersection numbers of two-step flag varieties, which leads to complicated combinatorics for calculating general structure constants.  There is an expression of the structure constants of degree one in terms of combination of structure constants of degree zero for most of flag varieties of general Lie type in \cite[Theorem 6.8]{LiMi}, which unfortunately  {involves sign cancelations}. As from the computational examples, it seems that the sufficient condition we provided   covers most of the non-vanishing terms $N_{\lambda, \mu}^{\nu, 1}$ already.

  By more careful but a bit tedious analysis for $Gr(3, n)$, we obtain a further application as follows, and refer to \textbf{Theorem \ref{thmQLRforGr3n}} for  detailed descriptions.
\begin{thm} In $QK(Gr(3, n))$, for any partitions $\lambda, \mu, \nu$ and  degree $d$,
we have     $$N_{\lambda, \mu}^{\nu, d}=N_{\tilde\lambda, \tilde\mu}^{\tilde \nu, 0}$$
 with precise descriptions of    partitions  $\tilde\lambda, \tilde\mu, \tilde \nu$,   except for a few cases for which a
   formula of $N_{\lambda, \mu}^{\nu, d}$ with manifestly alternating positivity can be provided directly.
\end{thm}
\noindent We therefore obtain a quantum Littlewood-Richardson rule for $QK(Gr(3, n))$, thanks to Buch's classical Littlewood-Richardson rule for $K(Gr(k, n))$ \cite{Buch}. As a direct consequence of our reduction and the well-known alternating positivity \cite{Buch,Brio} for
the Littlewood-Richardson coefficients for $K(Gr(k, n))$, we obtain the following.
\begin{cor}
    In $QK(Gr(3, n))$, for any partitions $\lambda, \mu, \nu$ and  degree $d$, we have
    $(-1)^{|\lambda|+|\mu|+|\nu|+nd}N_{\lambda, \mu}^{\nu, d}\geq 0$.
\end{cor}
\noindent  {The conjecture on the positivity of the structure constants \cite[Conjecture 5.10]{BuMi} was recently proved for minuscule Grassmannians and quadric hypersurfaces by Buch, Chaput, Mihalcea and Perrin with  a geometric method   \cite{BCMP22}.}
The above corollary  provides  an alternative proof of  the conjecture  in the special case $Gr(3, n)$.
We remark that the case $Gr(2, n)$ is trivial in the sense that the quantum product of general Schubert classes are directly reduced to the classical product of two special Schubert classes by Seidel operators.
For $Gr(k, n)$ with $k>3$, quite a few quantum Littlewood-Richardson coefficients cannot be reduced to classical ones by using all the known symmetries, even for the quantum cohomology.

The present paper is organized as follows. In section 2, we review Seidel representation for $QH^*(X)$. In section 3, we review basic facts on quantum $K$-theory of $X$. In section 4, we study projected Gromov-Witten varieties and provide a direct proof of  Seidel representation on $QK(X)$. We give applications of Seidel representations in the rest.
In section 5, we reprove the $K$-theoretic quantum Pieri rule by Buch and Mihalcea.
In section 6, we prove Theorem \ref{mainthm1} and show quantum-to-classical for certain structure constants of higher degree.
 Finally in section 7,  we provide a quantum Littlewood-Richardson rule for $QK(Gr(3, n))$.  We remark that  sections 6 and 7 are purely of combinatorial nature, which are self-contained after assuming   little background on $QK(X)$ together with the quantum Pieri rule.

\subsection*{Acknowledgements}


The authors would like to thank    Prakash Belkale, Boris Shapiro, Michael Shapiro and Lei Song for helpful discussions.  The authors are extremely grateful to    Leonardo C. Mihalcea
 for his valuable comments. The authors would also like to thank the anonymous reviewers for the careful reading of our manuscript and  their valuable  comments.

The first author would like to express the deepest gratitude to Prof. Bumsig Kim (1968-2021) for his mentorship. As suggested by Prof. Bumsig Kim, the first author learned  both quantum K-theory and Seidel elements for the first time, while he was a postdoc at KIAS.
The first author  was  supported by   the National Key Research and Development Program of China No.~2023YFA100980001 and NSFC Grant 12271529.
  \section{Seidel action on $QH^*(Gr(k, n))$}
 In this section, we review Seidel representation on the quantum cohomology $QH^*(X)$ of Grassmannians $X=Gr(k, n)$, mainly following \cite{Belk}. We refer to \cite{BuchGr, FuWo} for the well-known facts for $QH^*(X)$.

\subsection*{Quantum cohomology} Let $k, n$ be integers with $1\leq k<n$.
The complex Grassmannian   $X=\{V\leqslant \mathbb{C}^n\mid \dim V=k\}$ is a smooth projective variety.
Let $\mathcal{P}_{k, n}=\{\lambda=(\lambda_1, \cdots, \lambda_k)\in \mathbb{Z}^k\mid n-k\geq \lambda_1\geq \cdots \geq \lambda_k\geq 0\}$.
 Fix an arbitrary complete flag $F_\bullet=(F_1\leqslant F_2\leqslant\cdots\leqslant F_{n-1}\leqslant F_n=\mathbb{C}^n)$.
For each partition $\lambda\in \mathcal{P}_{k, n}$,  the  Schubert subvariety  (associated to the complete flag $F_\bullet$)
     $$X^\lambda=X^\lambda(F_\bullet):=\{V\in X\mid \dim V\cap F_{n-k+i-\lambda_i}\geq i,\,\, 1\leq i\leq k\}.$$
 is of   codimension $|\lambda|=\sum_{i=1}^k\lambda_i$. In particular, $X^{(1, \cdots, 1)}$ is   smooth, isomorphic to $Gr(k, n-1)$, and $X^{(n-k, 0, \cdots,0)}$ is also smooth, isomorphic to $Gr(k-1, n-1)$.

 By abuse of notation, we denote both  the homology class of $X^\lambda$ and its Poincar$\acute{e}$ dual cohomology class as $[X^\lambda]$.   The Schubert (cohomology) classes $[X^\lambda]$ are independent of choices of   $F_\bullet$, and  form a basis of  the cohomology of $X$:
     $H^*(X, \mathbb{Z})=\bigoplus_{\lambda\in \mathcal{P}_{k, n}}\mathbb{Z}[X^\lambda]$.
     The special  Schubert classes  $[X^j]:=[X^{(j, 0,\cdots,0)}]$
   (resp. $[X^{1^i}]: =[X^{(1,\cdots,1,0,\cdots, 0)}]$) generate the cohomology ring $(H^*(X, \mathbb{Z}), \cup)$. Moveover,
   $[X^j]=c_j(\mathcal{Q})$  and  $[X^{1^i}]=(-1)^ic_i(\mathcal{S})$  are the Chern classes of   tautological vector bundles. Here
    $0\to \mathcal{S}\to \mathbb{C}^n\to \mathcal{Q}\to 0$ is an exact sequence of vector bundles over $X$, in which   the fiber $\mathcal{S}|_{V\in X}$ is   the vector space $V$.
  We notice that $\int_X[X^\lambda]\cup[X^{\mu^\vee}]=\delta_{\lambda, \mu}$ for any    $\lambda, \mu\in \mathcal{P}_{k, n}$; here
  $\mu^\vee$ is the dual partition of $\mu=(\mu_1, \cdots, \mu_k)$, defined by $\mu^\vee:=(n-k-\mu_k, \cdots, n-k-\mu_1)$.

 Denoting by $\square$ the (unique) homology class of Schubert variety of complex dimension one, we have $H_2(X, \mathbb{Z})=\mathbb{Z}\square$. Let  $\overline{\mathcal{M}}_{0, 3}(X, d)$ be the moduli space of  three-pointed stable maps to $X$ of genus zero and degree $d\square$, which is of dimension $(\dim X+ nd)$. Let  $\mbox{ev}_i: \overline{\mathcal{M}}_{0, 3}(X, d)\to X$ denote the $i$-th evaluation map.
  The (small) quantum cohomology $QH^*(X)=(H^*(X, \mathbb{Z})\otimes\mathbb{Z}[q], \star)$ is deformation of $H^*(X, \mathbb{Z})$. The structure constants in the quantum product
     $$[X^\lambda]\star[X^\mu]=\sum_{d\in \mathbb{Z}_{\geq 0}}\sum_{\nu\in \mathcal{P}_{k, n}}c_{\lambda, \mu}^{\nu, d}[X^\nu]q^d$$
    are given by genus zero, three-pointed Gromov-Witten invariants:
    {\upshape $$ c_{\lambda, \mu}^{\nu, d}=\langle [X^\lambda], [X^\mu], [X^{\nu^\vee}]\rangle_d=\int_{\overline{\mathcal{M}}_{0, 3}(X, d)}\mbox{ev}_1^*[X^\lambda]\cup\mbox{ev}_2^*[X^\mu]\cup\mbox{ev}_3^*[X^{\nu^\vee}].$$
    }
   In particular, the above sum is finite, since
 $$
     c_{\lambda, \mu}^{\nu, d}\neq 0\quad\mbox{only if}\quad |\lambda|+|\mu|=|\nu|+nd.$$
  Moreover, $ c_{\lambda, \mu}^{\nu, d}\in \mathbb{Z}_{\geq 0}$, since geometrically it counts the number of morphisms
$f:\mathbb{P}^1\to X$ of degree $f_*([\mathbb{P}^1])=d\square$ with the properties $f(0)\in g_1X^\lambda, f(1)\in g_2X^\mu$ and
     $f(\infty)\in X^{\nu^\vee}$ for (fixed) generic $g_1, g_2\in GL(n, \mathbb{C})$.
\subsection*{Seidel representation on $QH^*(X)$} In addition to the 
set of partitions, Schubert classes in $H^*(X, \mathbb{Z})$ can  be indexed by other combinatorial sets, for instance by   the set  ${[n]\choose k}$ of jump sequences $1\leq a_1<a_2<\cdots<a_k\leq n$. There is a   bijection  $\varphi: \mathcal{P}_{k,n}\overset{\cong}{\longrightarrow} {[n]\choose k}$ of sets, given by $\lambda\mapsto I_\lambda=\{n-k+j-\lambda_j\mid 1\leq j\leq k\}$.

For $I=\{a_j\}_j\in {[n]\choose k}$ and $p\in \mathbb{Z}$, by
 $I+p$ we mean the unique element $\{b_j\}_j$ in ${[n]\choose k}$ such that for some permutation $\phi\in S_k$, $b_{\phi(j)}\equiv a_j+p \mod n$ for all $1\leq j\leq k$; for $0\leq i\leq n$, we denote by  $d_i(I)$ the cardinality of the set $\{j\mid a_j\leq i, 1\leq j\leq k\}$.
Following \cite[section 3.5]{BuWa}, we define the notion of\textit{ Seidel shift }as below.
\begin{defn}\label{Seidelshifts}
  {\upshape  We define the first  \textit{Seidel shift} of a partition $\lambda\in \mathcal{P}_{k, n}$  by
$$\lambda\uparrow1=\begin{cases} (\lambda_1+1, \lambda_2+1, \cdots, \lambda_k+1),&\mbox{if }\lambda_1<n-k,\\
   (\lambda_2,\lambda_3, \cdots, \lambda_m, 0),&\mbox{if }\lambda_1=n-k.
\end{cases}$$
The $p$-th Seidel shift $\lambda\uparrow p$ is defined inductively by $\lambda\uparrow 0=\lambda$ and  $\lambda\uparrow(p + 1) =(\lambda\uparrow p)\uparrow 1$ for $p\geq 0$. For $0\leq p\leq n$, we write $\lambda\downarrow p:=\lambda\uparrow (n-p)$, and notice
 \begin{equation}\label{shiftdual}
    (\lambda\uparrow p)^\vee=\lambda^\vee\downarrow p.
 \end{equation}
}
\end{defn}
\noindent We notice
$I_\lambda-1=I_{\lambda \uparrow 1}$ and $d_1(I_\lambda)={k+|\lambda|-|\lambda\uparrow 1|\over n}$,    following immediately from   
$\varphi$.

Seidel representation on $QH^*(X)$ is generated by the operator $T=[X^{1^k}]\star$ on $QH^*(X)$ 
in \cite{Belk},  which  maps a Schubert class $[X^\lambda]=[X^{I_\lambda}]$ to
 \begin{eqnarray}\label{SeidelQH}
     T([X^{I_\lambda}])=q^{d_1(I_\lambda)}[X^{I_\lambda-1}]=q^{k+|\lambda|-|\lambda\uparrow 1|\over n} [X^{\lambda\uparrow 1}].
 \end{eqnarray}
 Together with the fact $T^r([X^{I_\lambda}])=q^{d_r(I_\lambda)}[X^{I_\lambda-r}]$, it follows that
  \begin{eqnarray}\label{SeidelQHdlambdak}T^r([X^{\lambda}])=q^{d_r(I_\lambda)}[X^{\lambda\uparrow r}] \quad\mbox{with}\quad d_r(I_\lambda)={1\over n}\big(rk+|\lambda|-|\lambda\uparrow r|\big). \end{eqnarray}
As an application of the Seidel representation, Belkale obtained the following   in  the proof of  \cite[Theorem 10]{Belk}, which is stronger than the statement of his theorem.
\begin{prop}\label{Belformula}Let $\lambda, \mu \in \mathcal{P}_{k, n}$.
 The smallest power $d_{\rm min}$ of $q$ appearing in  $[X^\lambda]\star [X^\mu]$ in $QH^*(X)$ is the number
$$d_{\rm min} = \max\{{1\over n}\big({|\lambda|-|\lambda\uparrow i|+|\mu|-|\mu\uparrow (n-i)|} \big)| 0\leq i\leq  n\}.$$ Moreover, if the max
is achieved for $r$, then
   $$[X^\lambda]\star [X^\mu]=q^{d_{\rm min}} [X^{\lambda\uparrow r}]\star [X^{\mu\uparrow (n-r)}]. $$
\end{prop}
\begin{remark}
  The above proposition was originally given in terms of jump sequences, and is equivalent to \cite[Corollary 6.2 and the $D_{\rm min}$-part of Theorem 7.1]{Post}.
\end{remark}
\begin{remark}
   The formula $T([X^\lambda])$ is a special case of Bertram's quantum Pieri rule \cite{Bert}, but does not rely on it.
  The quantum Pieri rule was reproved by using Seidel representation and  the dimension constraint for  $c_{\lambda, \mu}^{\nu, d}$ \cite{Belk}.
\end{remark}
\section{Quantum $K$-theory of $Gr(k, n)$}

\subsection*{$K$-theory} The $K$-theory $K(X)=K^0(X)$, as a  free abelian group, is the Grothendieck group  of isomorphism classes $[E]$ of algebraic vector bundles over $X$ of finite rank, subject to the relations $[E]+[F]=[H]$ whenever there is a short exact sequence of algebraic vector bundles
 $0\to E\to H\to F\to 0$. The two ring operations on $K(X)$ are defined by
$[E]+ [F]:= [E\oplus F]$ and $[E]\cdot[F]:= [E \otimes F]$. Each structure sheaf $\mathcal{O}_{X^\lambda}$ has a resolution
 $0\to E_N\to E_{N-1}\to\cdots\to E_0\to \mathcal{O}_{X^\lambda}\to 0$ by locally free sheaves, making it meaningful to define the class $\mathcal{O}^\lambda=[\mathcal{O}_{X^\lambda}]:=\sum_{j=0}^N(-1)^j[E_j]\in K(X)$. Moreover, we have $K(X)=\bigoplus_{\lambda\in \mathcal{P}_{k, n}}\mathbb{Z}\mathcal{O}^\lambda$. The Euler characteristic of $[E]\in K(X)$ is given by
 $\chi_X([E]):=\sum_{j=0}^{\dim X}(-1)^j\dim H^i(X, E)$.
  $K(X)$ has another basis of ideal sheaves $\xi_\lambda$, satisfying $\chi_X(\mathcal{O}^\lambda\cdot \xi_\mu)=\delta_{\lambda, \mu}$ for all $\lambda, \mu\in \mathcal{P}_{k, n}$. Precisely,
      $\xi_\mu=[\mathcal{O}_{X^{\mu^\vee}}(-\partial X^{\mu^\vee})] =\sum_{\eta\supset \mu^\vee; \eta/\mu^\vee \rm is\,\, a\,\, rook\,\, strip}(-1)^{|\eta/\mu^\vee|}\mathcal{O}^\eta$.
The structure constants $N_{\lambda, \mu}^\nu$ in the product   satisfy the alternating positivity \cite{Buch,Brio}:
   $$\mathcal{O}^\lambda\cdot \mathcal{O}^\mu=\sum_{\nu\in \mathcal{P}_{k, n}}N_{\lambda, \mu}^{\nu} \mathcal{O}^\nu\quad\mbox{with}\quad (-1)^{|\lambda|+|\mu|+|\nu|}N_{\lambda, \mu}^{\nu}\in \mathbb{Z}_{\geq 0}.$$
    A Littlewood-Richardson rule for them was first given by Buch \cite{Buch}.

\subsection*{Quantum $K$-theory}
 Similar to the cohomological Gromov-Witten invariants, the genus zero $K$-theoretic Gromov-Witten invariants of degree $d\square$ for $\gamma_1, \cdots, \gamma_m\in K(X)$ are defined by\footnote{For Grassmannian $X$, the moduli space $\overline{\mathcal{M}}_{0, m}(X, d)$ of stable maps is an orbifold, so that the issue of virtual structure sheaf will not be involved.}
   $$I_d(\gamma_1, \cdots, \gamma_m):=\chi_{\overline{\mathcal{M}}_{0, m}(X, d)}(\mbox{ev}_1^*(\gamma_1)\cdot \cdots\cdot \mbox{ev}_m^*(\gamma_m))\in \mathbb{Z}.$$
  The (small) quantum $K$-theory $QK(X)=(K(X)\otimes \mathbb{Z}[[q]], *)$ is a deformation of the classical $K$-theory $K(X)$. The  structure constants in the  quantum
   $K$-product,
    $$\mathcal{O}^\lambda* \mathcal{O}^\mu=\sum_{d\in \mathbb{Z}_{\geq 0}}\sum_{\nu\in \mathcal{P}_{k, n}}N_{\lambda, \mu}^{\nu, d} \mathcal{O}^\nu q^d,$$
   are defined by combination of $K$-theoretic Gromov-Witten invariants, or equivalently defined recursively by structure constants of smaller degree as follows \cite{Give, BuMi}.
   \begin{eqnarray}
    \label{sc1}{}\hspace{1cm}  N_{\lambda, \mu}^{\nu, d}
      &=&\sum_{r\in \mathbb{Z}_{\geq 0}}\sum\limits_{(d_0, \cdots,d_r)}\hspace{0cm}\sum\limits_{\kappa_1, \cdots, \kappa_r\in \mathcal{P}_{k, n}}\hspace{-0.3cm}(-1)^rI_{d_0}(\mathcal{O}^\lambda, \mathcal{O}^\mu, \xi_{\kappa_1})\prod_{i=1}^{r}I_{d_i}(\mathcal{O}^{\kappa_i},\xi_{\kappa_{i+1}})\\
     \label{sc2} &=&I_d(\mathcal{O}^\lambda, \mathcal{O}^\mu, \xi_{\nu})-\sum_{\kappa\in \mathcal{P}_{k, n};0<e\leq d}N_{\lambda, \mu}^{\kappa, d-e}I_e(\mathcal{O}^\kappa,\xi_\nu).
   \end{eqnarray}
  The sum in \eqref{sc1}  is over all   $(d_0,\cdots , d_r )\in \mathbb{Z}_{\geq 0}\times \mathbb{Z}^r_{>0}$ with $d=\sum_{i=0}^rd_i$, and   we set $\kappa_{r+1} = \nu$.
In contrast to the cohomological  invariants, the $K$-theoretic Gromov-Witten invariants  could be nonzero even for large $d$. Nevertheless, the structure constants $N_{\lambda, \mu}^{\nu, d}$ do  vanish whenever $d$ is large enough \cite{BuMi,BCMPfinite00,BCMPfinite, ACT}. Moreover, $N_{\lambda, \mu}^{\nu, 0}=N_{\lambda, \mu}^{\nu}$, namely $\mathcal{O}^\lambda*\mathcal{O}^\mu|_{q=0}=\mathcal{O}^\lambda\cdot\mathcal{O}^\mu \in K(X)$.

\section{Seidel action on $QK(Gr(k, n))$ 
 }
 In this section, we provide a direct proof of   Seidel representation on $QK(X)$, by describing projected Gromov-Witten varieties precisely.

 \subsection{Projected Gromov-Witten varieties}
Recall $X^\lambda=X^\lambda(F_\bullet)$ is of codimension $|\lambda|$. There are also Schubert varieties $X_\lambda$ of dimension $|\lambda|$, opposite to $X^\lambda$. Precisely, we consider the complete flag $F^{\rm opp}_\bullet :=C_{w_0} F_\bullet$,
where $C_{w_0}$ denotes the permutation matrix associated to the longest permutation $w_0\in S_n$. That is,
   $F_i^{\rm opp}$ is spanned by $\{e_n, e_{n-1}, \cdots, e_{n-i+1}\}$ for all $i$.
 Then $X_\lambda:=C_{w_0}\cdot X^{\lambda^\vee}$ is given by
 $$X_\lambda=X_\lambda(F^{\rm opp}_\bullet):=\{\Lambda\in X\mid \dim \Lambda\cap F^{\rm opp}_{i+\lambda_{k+1-i}}\geq i, 1\leq i\leq k\} =C_{w_0}\cdot X^{\lambda^\vee}.$$
 Recall that $\mbox{ev}_i: \overline{\mathcal{M}}_{0,3}(X, d)\to X$ is the $i$-th evaluation map.
\begin{defn}{\upshape
 Let $Y, Z\subset X$ be closed subvarieties.  The Gromov-Witten variety of $Y$ and $Z$ 
 of degree $d$ is defined by $M_d(Y, Z):= {\rm ev}_1^{-1}(Y)\cap {\rm ev}_2^{-1}(Z)\subset \overline{\mathcal{M}}_{0,3}(X, d)$.
   The image
        $\Gamma_d(Y, Z):={\rm ev}_3(M_d(Y, Z))$ is called a  \textit{projected Gromov-Witten variety}.
   }
\end{defn}

Let $F\ell_{n_1,\cdots, n_r; n}:=\{V_{n_1}\leqslant V_{n_2}\leqslant \cdots\leqslant V_{n_r}\leqslant \mathbb{C}^n\mid \dim V_{n_i}=n_i, 1\leq i\leq r\}$.
For $1\leq d< \min\{k+1, n-k\}$, we consider the natural projections  $\pi_{\rm G}$, $\pi_{\rm F}$, $pr_1$ and $pr_2$   among flag varieties in the following commutative diagram.
\begin{equation}\label{diag}
   \xymatrix{
    & Z_d:=F\ell_{k-d, k, k+d; n}   \ar[0,0];[1,1]_{\pi_G } \ar[0,0];[1,-1]^{\pi_F }\ar[r]^{{}\qquad pr_1} &  F\ell_{k, k+d; n}  \ar[d]^{pr_2}\\
  F\ell_{k-d, k+d; n}   &    & X=Gr(k, n)    }
\end{equation}
Geometrically, $\Gamma_d(Y, Z)$ consists of points $p$ in $X$ such that there exists a rational curve $C_d$ of degree $d$ in $X$ that passes $p$, $Y$ and $Z$.
For $d=1$, the composition $\pi_{\rm G}\circ \pi_{\rm F}^{-1}$ gives an isomorphism
    between $F\ell_{k-1, k+1; n}$ and the space of lines in $Gr(k, n)$ \cite{Stri}. For general $d$,  the
kernel $V_{k-d}$ and span $V_{k+d}$ \cite{Buch} of a general rational curve $C_d\subset X$ form a point  $V_{k-d}\leqslant V_{k+d}$
 in $F\ell_{k-d, k+d; n}$.
 By chasing the  commutative diagram (8) in \cite{BuMi} (see for instance \cite[formula (4)]{BCMP11} for the special case  $\Gamma_d(X^{\lambda}, X_{\eta})$), for any subvarieties $Y_1, Y_2$ of $X$, we have the following formula
 \begin{eqnarray}
 \label{GammadY1Y2}   \Gamma_d(Y_1, Y_2)&=& \pi_{\rm G}(Z_d(Y_1, Y_2)),\\
 \label{ZdY1Y2} \mbox{ where } \quad  Z_d(Y_1, Y_2)&=& \pi^{-1}_{\rm F}\big(
   \pi_{\rm F}\pi_{\rm G}^{-1}(Y_1)\bigcap \pi_{\rm F}\pi_{\rm G}^{-1}(Y_2)\big).
 \end{eqnarray}
In particular when $Y_2=X$, we simply denote
 \begin{eqnarray}
 \label{GammadZdY1}   \Gamma_d(Y_1):=\Gamma_d(Y_1, X)\quad\mbox{and}\quad Z_d(Y_1):=Z_d(Y_1, X).
 \end{eqnarray}

The restriction $\mbox{ev}_3: M_d(X^\lambda, X_\eta)\to \Gamma_d(X^\lambda, X_\eta)$ is   cohomological trivial by \cite[Theorem 4.1 (b)]{BCMP11}. The following is a direct consequence, where $\mathcal{O}_\eta:=[\mathcal{O}_{X_\eta}]$.
\begin{prop}[Corollary 4.2 of \cite{BCMP11}] For any $\lambda, \eta\in \mathcal{P}_{k, n}$, $\gamma\in K(X)$ and $d\geq 1$, 
     \begin{eqnarray}
  \label{KGW3} I_d(\mathcal{O}^{\lambda},\mathcal{O}_\eta, \gamma)&=&\chi_X([\mathcal{O}_{\Gamma_d(X^\lambda, X_\eta)}]\cdot \gamma).
 \end{eqnarray}
\end{prop}

 \subsection{Main propositions} In this subsection, we provide concrete descriptions of relevant projected Gromov-Witten varieties, which play key roles in our direct proof of Seidel representation as well as  in the alternative proof of quantum Pieri rule.


\begin{lemma}\label{mainlemmapieri}
  Let $\eta\in \mathcal{P}_{k,n}$, $1\leq i\leq n-k$ and  $1\leq d<n-k$. Let $V_k\in Gr(k, n)$ satisfy
   $\dim V_k \cap (F^{\rm opp}_{d+\eta_{k-d+1}}+F_{n-k-i+1})\geq 1$ and   for all  $d< j\leq k$,
  \begin{align*}
   \dim V_k\cap F^{\rm opp}_{j+\eta_{k-j+1}}\geq j-d, \quad
  \dim V_k \cap (F^{\rm opp}_{j+\eta_{k-j+1}}+F_{n-k-i+1})\geq j-d+1.
\end{align*}
Then there exists   $V_{k+d}\in Gr(k+d, n)$ with $V_k\leq V_{k+d}$ that satisfies
   \begin{align*}
  \dim V_{k+d}\cap F_{n-k-i+1}\geq 1,\quad \dim V_{k+d}\cap F^{\rm opp}_{s+\eta_{k-s+1}}\geq s, \,\, \forall\,\,1\leq s\leq k.
   \end{align*}\end{lemma}
\begin{proof}
We consider the set  $R:=\bigcup_{d\leq j\leq k}S_j$ with each $S_j$ being defined by
 \begin{align*} S_j:=\{v\mid v=w+v_0 \mbox{ for some } w\in F^{\rm opp}_{j+\eta_{k-j+1}} \mbox{ and } v_0\in F_{n-k-i+1}\setminus\{0\};
    v\in V_k\setminus\{0\}
   \}.
   \end{align*}

Assume $R=\emptyset$ first.  Let $\{v_1,\cdots,v_k\}$ be a set of linearly independent vectors satisfying $v_s\in F^{\rm opp}_{s+\eta_{k-s+1}}$ for all $1\leq s\leq k$.
 Take   $v_0\in F_{n-k-i+1}\setminus\{0\}$. Clearly, there is $d-1\leq r\leq k$ such that
  $V_{k+d}:=V_k+H_r$ is of dimension $k+d$, where $H_r:=\mathbb{C}v_0+\mathbb{C}v_1+\cdots +\mathbb{C}v_r$.
  By definition, we have $v_0\in V_{k+d}\cap   F_{n-k+1-i}\setminus\{0\}$ and
  $\dim V_{k+d}\cap F^{\rm opp}_{s+\eta_{k-s+1}} \geq \dim H_{r}\cap F^{\rm opp}_{s+\eta_{k-s+1}}\geq  s$   for  $ 1\leq s\leq r$.
Note $V_{k+d}=V_k+\tilde H_d$ for some $\tilde H_d\subset H_r$ with $V_k\cap \tilde H_d=\emptyset.$
For $k\geq s>r\geq d-1$, it follows from $S_s=\emptyset$ that
$\dim V_{k}\cap F^{\rm opp}_{s+\eta_{k-s+1}}=\dim V_{k}\cap (F^{\rm opp}_{s+\eta_{k-s+1}}+F_{n-k+1-i})\geq s-d+1$;
hence $\dim V_{k+d}\cap  F^{\rm opp}_{s+\eta_{k-s+1}} \!\!\geq \dim V_{k}\cap  F^{\rm opp}_{s+\eta_{k-s+1}} +
 \dim \tilde H_{d}\cap  F^{\rm opp}_{s+\eta_{k-s+1}}\!\! \geq s-d+1 + (d-1)\!=s.$

 Now we assume   $R\neq\emptyset$, and set $m:=\min\{j\mid S_j\neq \emptyset; d\leq j\leq k\}$. Then
 there exists $v=w+v_0$ with $w\in F^{\rm opp}_{m+\eta_{k-m+1}}$, $v_0\in F_{n-k+1-i}\setminus\{0\}$ and $v\in V_k\setminus\{0\}$.
   Let $\{v_1,\cdots,v_k\}$ be a set of linearly independent vectors satisfying $v_s\in F^{\rm opp}_{s+\eta_{k-s+1}}$ for all $1\leq s\leq k$ and
   $w\in   \mathbb{C}v_1+\cdots +\mathbb{C}v_m$. Set $\hat H_0:=\mathbb{C}w$ and $\hat H_j:=\mathbb{C}w+\mathbb{C}v_1+\cdots +\mathbb{C}v_j$ for $j>0$.
 Once again we have  $V_{k+d}:=V_k+\hat H_r\in Gr(k+d, n)$
for some $d-1\leq r\leq k$, and  $\dim V_{k+d}\cap F^{\rm opp}_{s+\eta_{k-s+1}} \geq \dim \hat H_{r}\cap F^{\rm opp}_{s+\eta_{k-s+1}}\geq  s$
for  $ 1\leq s\leq r$.
Moveover, $0\neq v_0=v-w\in V_{k}+\hat H_{r}=V_{k+d}$, implying $\dim V_{k+d}\cap F_{n-k+1-i}\geq 1$.

It remains to discuss   $r<s\leq k$.
Let $t:=\dim \hat H_r$. Clearly,  $r\leq t\leq r+1$.
  If $m\leq r$,   by the same arguments for $R=\emptyset$, we conclude $\dim V_{k+d}\cap  F^{\rm opp}_{s+\eta_{k-s+1}}\geq s$ for $r<s\leq k$.
  If  $r<m$, then we still have
    \begin{align*}
    &\dim V_{k+d}\cap F^{\rm opp}_{s+\eta_{k-s+1}}\\
     \geq& \dim V_k\cap F^{\rm opp}_{s+\eta_{k-s+1}} + \dim \hat H_r\cap F^{\rm opp}_{s+\eta_{k-s+1}}-\dim V_{k}\cap \hat H_r\\
   \geq & \begin{cases}
      (s-d+1)+r-(k+t-(k+d))=s+1+r-t\geq s, &\mbox{for } r<s<m,\\
      (s-d)+t-(k+t-(k+d))=s, &\mbox{for }r<m\leq s\leq k.
  \end{cases}
\end{align*}
Here $\dim V_k\cap F^{\rm opp}_{s+\eta_{k-s+1}}\geq s-d+1$  follows from $S_s=\emptyset$ for all $r\leq s<m$.
\end{proof}

\begin{lemma}\label{genZd}
Let  $\eta\in \mathcal{P}_{k, n}$ and $1\leq i\leq n-k$. For $1\leq d<\min\{k+1, n-k\}$, 
  \begin{align*}
      Z_d(X^i, X_\eta)&= \{V_{k-d}\leq V_k \leq V_{k+d}\mid  \dim V_{k+d}\cap  F^{\rm opp}_{s+\eta_{k-s+1}} \geq s, \mbox{ for   } 1\leq s\leq k; \\
                    \dim & V_{k+d}\cap   F_{n-k+1-i} \geq 1;\,\, \dim V_{k-d}\cap F^{\rm opp}_{j+\eta_{k-j+1}} \geq j-d, \mbox{ for   } d+1\leq j\leq k  \}.
                         \end{align*}
\end{lemma}
\begin{proof}
By definition, we have  $\pi^{-1}_{\rm F}\pi_{\rm F}\pi_{\rm G}^{-1}(X_\eta)=\mbox{RHS}1$ with
       \begin{align*}
       \mbox{RHS}1
       & := \{V_{k-d}\leq V_k \leq V_{k+d}\mid   \exists \bar{V_k} \mbox{ such that }  V_{k-d}\leq \bar{V_k}\leq V_{k+d}, \\
                  &{}\hspace{4cm}  \dim \bar{V_k}\cap F^{\rm opp}_{s+\eta_{k-s+1}} \geq s, \mbox{ for  } 1\leq s\leq k \}.
        \end{align*}
Clearly, we have $\mbox{RHS}1\subseteq \mbox{RHS}2$, where
   \begin{align*}
      \mbox{RHS}2
       &  : {=} \{V_{k-d}\leq V_k \leq V_{k+d}\mid  \dim V_{k+d}\cap  F^{\rm opp}_{s+\eta_{k-s+1}} \geq s, \mbox{ for   } 1\leq s\leq k, \\
                     &{}\hspace{3.6cm}  \dim V_{k-d}\cap F^{\rm opp}_{j+\eta_{k-j+1}} \geq j-d, \mbox{ for   } d+1\leq j\leq k  \}.
    \end{align*}
     To show $\mbox{RHS}1\supseteq \mbox{RHS}2$, we denote $G_s:=F^{\rm opp}_{s+\eta_{k-s+1}}$ for all $1\leq s\leq k$. Since $\dim V_{k+d}\cap G_s\geq s$ for $1\leq s\leq d$, we can take a partial flag $W_1\leqslant W_2\leqslant \cdots \leqslant W_d$ with $W_s\subset V_{k+d}\cap G_s$ for all $1\leq s\leq d$. Clearly $\dim W_d\cap V_{k-d}\geq 0=d-d$. Let $d\leq r<k$. Assume that we have obtained partial flag $W_d\leqslant W_{d+1}\leqslant \cdots \leqslant W_r$ that satisfies
   $W_j\subset V_{k+d}\cap G_j$ and $\dim W_j\cap V_{k-d}\geq j-d$ for all $d\leq j\leq r$. Notice  $\dim V_{k-d}\cap G_{r+1}\geq r+1-d$ and $\dim V_{k+d}\cap G_{r+1}\geq r+1$. If there exists $v\in \big(V_{k-d}\cap G_{r+1}\big)\setminus W_r$, then we take $W_{r+1}=\mathbb{C}v+W_r$; otherwise, we take any $w\in \big(V_{k+d}\cap G_{r+1}\big)\setminus W_r$ and set $W_{r+1}=\mathbb{C}w+W_r$. In either cases, we have
  $W_{r+1}\subset V_{k+d}\cap G_{r+1}$ and  $\dim W_{r+1}\cap V_{k-d}\geq r+1-d$. Thus by induction we obtain a partial flag $W_1\leqslant W_2\leqslant \cdots \leqslant W_k$ with the properties $W_k\subset V_{k+d}\cap G_k\subset V_{k+d}$,
      $\dim W_k \cap G_s\geq \dim W_s\cap G_s=\dim W_s=s$ for all $1\leq s\leq k$, and $\dim W_k\cap V_{k-d}\geq k-d$. The last inequality implies $V_{k-d}\subset W_k$.  Therefore, for any element $V_{k-d}\leqslant V_k\leqslant V_{k+d}$ satisfying the constraints in RHS$2$, we can find $W_k$ that satisfies the constraint in RHS$1$. Hence,  $\mbox{RHS}2\subseteq \mbox{RHS}1$.
     \begin{align*}
       \pi^{-1}_{\rm F}\pi_{\rm F}\pi_{\rm G}^{-1}(X^{i})&= \{V_{k-d}\leq V_k \leq V_{k+d}\mid \exists \bar{V_k} \mbox{ such that } V_{k-1}\leq \bar{V_k}\leq V_{k+1}, \bar{V_k}\in X^{i}\}\\
              &= \{V_{k-d}\leq V_k \leq V_{k+d}\mid \dim V_{k+d}\cap F_{n-k+1-i} \geq 1\}.
    \end{align*}
   Since $Z_d(X^i, X_\eta)= \pi^{-1}_{\rm F}\pi_{\rm F}\pi_{\rm G}^{-1}(X^{i})\cap \pi^{-1}_{\rm F}\pi_{\rm F}\pi_{\rm G}^{-1}(X_\eta)$, the statement follows.
   \end{proof}

 Set $\tilde F_\bullet:=g\cdot F^{\rm opp}_\bullet$, where    the permutation matrix $g$ is defined by
  \begin{equation}  \label{matrix g}
     g:=\left(
         \begin{array}{cccccc}
           0 & 1 & 0 & \cdots & \cdots & 0 \\
           0 & 0 & 1 & \ddots &  & \vdots \\
           \vdots &  & \ddots & \ddots & \ddots & \vdots \\
           \vdots &  & & \ddots & \ddots & 0 \\
           0& 0 & \cdots & \cdots  & 0 & 1 \\
           1 & 0 & \cdots & \cdots & 0 & 0 \\
         \end{array}
       \right).
  \end{equation}
Namely
 $\tilde F_1=\mathbb{C}e_1$, and $\tilde F_j$ is spanned by $\{e_1, e_n, \cdots, e_{n-j+2}\}$ for $2\leq j\leq n$.

\begin{prop}\label{deggeneralGamma}
Let   $\eta\in \mathcal{P}_{k, n}$ and $1\leq i\leq n-k$. For $1\leq d<\min\{k+1, n-k\}$, 
   \begin{align*}
      \Gamma_d(X^i, X_\eta)
       &  {=} \{V_k \mid  \dim V_{k}\cap  F^{\rm opp}_{j+\eta_{k-j+1}} \geq j-d, \mbox{ for   } d+1\leq j\leq k, \\
                     &{}\qquad \quad  \dim V_{k}\cap (F^{\rm opp}_{s+\eta_{k-s+1}}+F_{n-k+1-i}) \geq s-d+1, \mbox{ for   } d\leq s\leq k  \}.
    \end{align*}
    In particular, $\Gamma_d(X^{n-k}, X_{\eta})$ is a Schubert variety given by
     $$\Gamma_d(X^{n-k}, X_{\eta})=X_{\tilde \eta}(\tilde F_\bullet)= g\cdot X_{\tilde \eta},
   $$
   {where }$ \tilde \eta_j=
      \min\{\eta_{j-d+1}+d, n-k\}$ for $1\leq j\leq k$, with $\eta_i:=n$ for $i\leq 0$.

\end{prop}
\begin{proof}
       Denote by $\mbox{RHS}$  the right hand side of the equality in the statement.

Let $V_k\in \Gamma_d(X^i, X_\eta)=\pi_G(Z_d(X^i, X_\eta))$.
 Then there exists $V_{k-d}\leqslant V_k\leqslant V_{k+d}$ in $Z_d(X^i, X_\eta)$. By Lemma \ref{genZd}, we have
   have
 $$ \dim V_k\cap F^{\rm opp}_{j+\eta_{k+1-j}}\geq \dim V_{k-d}\cap F^{\rm opp}_{j+\eta_{k+1-j}}\geq j-d, \quad d+1\leq j\leq k $$
By the definition of $F_\bullet$ and $F_\bullet^{\rm opp}$,  either $F^{\rm opp}_{j+\eta_{k-j+1}}\cap F_{n-k-i+1}=0$
  or  $F^{\rm opp}_{j+\eta_{k-j+1}}+F_{n-k-i+1}=\mathbb{C}^n$ must hold. For any $1\leq s\leq k$, in the former case, we have
  $$\dim V_{k+d}\cap (F^{\rm opp}_{s+\eta_{k-j+1}}+F_{n-k-i+1})\geq V_{k+d}\cap F^{\rm opp}_{s+\eta_{k-j+1}}+\dim V_{k+d}\cap F_{n-k-i+1}\geq s+1;$$
in the latter case, $\dim V_{k+d} \cap (F^{\rm opp}_{s+\eta_{k-s+1}}+F_{n-k-i+1})=\dim V_{k+d}\geq s+1.$
 Therefore   $\dim V_{k} \cap (F^{\rm opp}_{s+\eta_{k-s+1}}+F_{n-k-i+1}) \geq s+1-d$ in either cases. Thus $\Gamma_d(X^i, X_\eta)\subseteq\mbox{RHS}$.

   Now we take any $V_k$ in RHS.
    By Lemma \ref{mainlemmapieri}, there exists $V_{k+d}\in Gr(k+d, n)$ with $V_k\leqslant V_{k+d}$ that satisfies
     $$\dim V_{k+d}\cap F_{n-k-i+1}\geq 1,\quad \dim V_{k+d}\cap F^{\rm opp}_{s+\eta_{k-j+1}}\geq s, \,\, \forall\,\,1\leq s\leq k.$$
Since  $\dim V_{k}\cap  F^{\rm opp}_{j+\eta_{k-j+1}} \geq j-d$ for $d<j\leq k$,
there exists  a partial flag $V_1\leq \cdots \leq V_{k-d}\leq V_k$ with
   $V_{j-d}\subset V_k\cap F^{\rm opp}_{j+\eta_{k-j+1}}$ for $d<j\leq k$. In particular,
   $\dim V_{k-d}\cap F^{\rm opp}_{j+\eta_{k-j+1}}\geq j-d$ for $d<j\leq k$.
Hence,  the element $V_{k-d}\leqslant V_k\leqslant V_{k+d}$ is in $Z_d(X^i, X_\eta)$. Thus $\Gamma_d(X^i, X_\eta)\supseteq\mbox{RHS}$.

So far we have shown $\Gamma_d(X^i, X_\eta)=\mbox{RHS}$. In particular for $i=n-k$, we have
   \begin{align*}
      \Gamma_d(X^{n-k}, X_\eta)
       &  {=} \{V_k \mid  \dim V_{k}\cap  F^{\rm opp}_{j+\eta_{k-j+1}} \geq j-d, \mbox{ for   } d+1\leq j\leq k, \\
                     &{} \dim V_{k}\cap (F^{\rm opp}_{s+\eta_{k-s+1}}+F_{n-k+1-(n-k)}) \geq s-d+1, \mbox{ for   } d\leq s\leq k  \}\\
        &  {=} \{V_k \mid  \dim V_{k}\cap (F^{\rm opp}_{s+\eta_{k-s+1}}+F_1) \geq s-d+1, \mbox{ for   } d\leq s\leq k  \}\\
        &  {=} \{V_k \mid  \dim V_{k}\cap \tilde F_{s-d+1+(d+\eta_{k-s+1})} \geq s-d+1, \mbox{ for   } d\leq s\leq k  \}.
    \end{align*}
 Hence, the last part follows directly from the definitions of $\tilde F_\bullet$ and $X_{\tilde \eta}(\tilde F_\bullet)$.
\end{proof}

  The curve neighborhood $\Gamma_d(\Omega)$ of a Schubert variety $\Omega$ in a general flag variety
  was studied in  \cite{BuMi00}. Therein it was shown  that
   $\Gamma_d(\Omega)$ is a Schubert variety labeled by the Hecke product of Weyl group elements. In the special case of (co)minuscule Grassmannians,
     $\Gamma_d(\Omega)$ was studied earlier in \cite{BCMPfinite00}.
  For $X=Gr(k, n)$, we have
  \begin{align}\label{GammaSch}
     \Gamma_d(X^\lambda)&=X^{\lambda^d}\textcolor{red}{,}
  \end{align}
   where the partition   $\lambda^d$
   is  obtained by removing the first $d$ rows and
columns from $\lambda$  \cite{BuMi00}. Equivalently, for all $j$,  $\lambda^{j+1}$  is obtained by removing the outer rim of the partition $\lambda^j$; namely for $\lambda^j=(\lambda_1^j,\cdots, \lambda_k^j)$,   $\lambda^{j+1}=(\lambda_2^j-1, \cdots, \lambda^j_{k}-1, 0)$ where $\lambda^j_i-1$ is replaced  by $0$ whenever $\lambda^j_i=0$  \cite[Theorem 2.5]{ShWi}. Here $\lambda^0:=\lambda$.

\begin{cor}\label{compareGamma}
  Let $\mu\in \mathcal{P}_{k, n}$. If $\mu_k>0$, then for any        $d\geq 1$, we have
   $$      \Gamma_d(X^{n-k}, X_{\mu^\vee})=
          \Gamma_{d-1}(g\cdot X_{\mu^\vee\uparrow 1}).
   $$
\end{cor}

\begin{proof}
Let $m=\min\{k+1, n-k\}$. For $1\leq d<m$, we notice $\mu^\vee_1=n-k-\mu_k<n-k$. By Proposition \ref{deggeneralGamma}, we have
   $\Gamma_1(X^{n-k}, X_{\mu^\vee})=X_{\tau^{(1)}}(\tilde F_\bullet)=g\cdot X_{\tau^{(1)}}$ with
  \begin{align*}
    \tau^{(1)}&=(\mu^\vee_1+1, \cdots, \mu^\vee_{k-d+1}+1)=\mu^\vee\uparrow 1.
 \end{align*}
  It follows directly from the description of $\Gamma_d(X^{n-k}, X_{\mu^\vee})=X_{\tau^{(d)}}(\tilde F_\bullet)$ in Proposition \ref{deggeneralGamma}
  that  $\big(\tau^{(d+1)}\big)^\vee$ is
  obtained  by removing the outer rim from $\big(\tau^{(d)}\big)^\vee$.
  Therefore the statement holds for $1\leq d<m$
  by using \eqref{GammaSch} and   induction on $d$.

   The partition  $(\mu^\vee\uparrow 1)$ has at most  $k$ rows and at most  $n-k-1$ columns.  Thus for any $d\geq m$,
   we have $\Gamma_{d-1}(g\cdot X_{\mu^\vee\uparrow 1})=X$ by \eqref{GammaSch}, and consequently
      $$X= \Gamma_{d-1}(\Gamma_1(X^{n-k},  X_{\mu^\vee}))\subseteq  \Gamma_{d}(X^{n-k},  X_{\mu^\vee})\subseteq X.$$
    \end{proof}

 Although $Z_d(X^i \cap X_\eta)$ is a proper subvariety of  $Z_d(X^i, X_\eta)$, they could have the same image under the projection $pr_1$, as shown in the following   key lemma.
\begin{lemma}\label{lemmiddle}
    Let $\eta\in \mathcal{P}_{k, n}$ and $1\leq i\leq n-k$. If $\eta_1=n-k$, then  we have
  \begin{align*}
 pr_1(Z_d(X^i, X_\eta))=pr_1(Z_d(X^i \cap X_\eta))\quad \mbox{ for any } 1\leq d<\min\{k, n-k\}.
 \end{align*}
\end{lemma}

\begin{proof}
    The direction $pr_1(Z_d(X^i, X_\eta))\supseteq pr_1(Z_d(X^i \cap X_\eta))$  follows immediately from the definitions in \eqref{ZdY1Y2} and \eqref{GammadZdY1}.
 It remains to show    that any two-step flag    $V_k\leqslant V_{k+d}$ in $pr_1(Z_d(X^i, X_\eta))$ must also belong to  $pr_1(Z_d(X^i\cap X_\eta))$.

 Since $V_{k-d}\leqslant V_k\leqslant V_{k+d}$ belongs to  $Z_d(X^i, X_\eta)$ for some $V_{k-d}$, by Lemma \ref{genZd} we have  
$\dim V_{k-d}\cap F^{\rm opp}_{j+\eta_{k-j+1}} \geq j-d \mbox{ for   } d+1\leq j\leq k$, and
  $$ \dim   V_{k+d}\cap   F_{n-k+1-i} \geq 1,\quad \dim V_{k+d}\cap  F^{\rm opp}_{s+\eta_{k-s+1}} \geq s, \mbox{ for   } 1\leq s\leq k.
                    $$
Hence, there exist $v_0\in V_{k+d}\cap F_{n-k-i+1}\setminus\{0\}$  and  a partial flag $\bar V_1\leq \cdots \leq \bar V_{k-1}\leq V_{k+d}$ with
   $\bar V_{s}\subset V_{k+d}\cap F^{\rm opp}_{s+\eta_{k-s+1}}$ for $1\leq s\leq k-1$. In particular,  we have
   $$\dim \bar V_{k-1}\cap F^{\rm opp}_{s+\eta_{k-s+1}}\geq s,\quad  1\leq s\leq k-1.$$
 If $v_0\notin \bar {V}_{k-1}$, then we define $\bar{V}_k:=\bar{V}_{k-1}+\mathbb{C}v_0$. Otherwise, we take $w\in V_{k+d}\setminus \bar{V}_{k-1}$ and then define $\bar{V}_k:=\bar V_{k-1}+\mathbb{C}w$.
 In either cases, we  have
  $$\dim \bar{V}_k\cap F_{n-k-i+1}\geq 1,\quad \dim \bar{V}_k\cap F^{\rm opp}_{s+\eta_{k-s+1}}\geq s,\,\,\, 1\leq s\leq k-1. $$
    Since $\eta_1=n-k$, $F^{\rm opp}_{k+\eta_{1}}=\mathbb{C}^n$. Thus $\dim \bar{V}_k\cap F^{\rm opp}_{k+\eta_{1}}=k$ and
      $\bar{V}_k\in X^i\cap X_\eta$.
      Since $V_k$ and $\bar{V_k}$ are both $k$-dimensional vector subspaces in $V_{k+d}$, we have $\dim V_k\cap \bar{V}_k \geq k-d$.
      Thus there exists  $(k-d)$-dimensional vector subspace  $\hat V_{k-d}\subset V_k\cap \bar{V_k}$, implying
      $$\dim \hat V_{k-d}\cap F^{\rm opp}_{j+\eta_{k-s+1}}\geq j-d,\quad d+1\leq j\leq k.$$
     Then
        $\hat V_{k-d}\leqslant  V_k \leqslant V_{k+d}$ belongs to $\pi_F^{-1}\pi_F\pi_G^{-1}(X^i\cap X_\eta)=Z_d(X^i\cap X_\eta)$, since the partial flag
      $\hat V_{k-d}\leqslant \bar V_k \leqslant V_{k+d}$ satisfies  $\bar V_k\in X^i\cap X_\eta$. Thus
        $V_k\leqslant V_{k+d}$ belongs to $pr_1(Z_d(X^i \cap X_\eta))$. Hence,
 $pr_1(Z_d(X^i, X_\eta))\subseteq pr_1(Z_d(X^i \cap X_\eta))$.
\end{proof}
\begin{prop}\label{Gammacomparemukzero}
  Let $\eta\in \mathcal{P}_{k, n}$ and $1\leq i\leq n-k$. If $\eta_1=n-k$, then  we have
  \begin{align*}
 \Gamma_d(X^i, X_\eta)=\Gamma_d(X^i \cap X_\eta)\quad \mbox{ for any } d\geq 1.
 \end{align*}
\end{prop}
\begin{proof}
For $1\leq d<\min\{k, n-k\}$, by   Lemma \ref{lemmiddle} and  $\pi_G=pr_2\circ pr_1$, we have
$$ \Gamma_d(X^i, X_\eta)= pr_2\circ pr_1(Z_d(X^i, X_\eta))=pr_2\circ pr_1(Z_d(X^i\cap X_\eta)) =\Gamma_d(X^i \cap X_\eta).$$
     Since $\eta_1=n-k$, $X^i\cap X_\eta\neq \emptyset$. Consequently for  $d\geq \min\{k, n-k\}$, we have
     $$X=\Gamma_d(\mbox{point})\subseteq \Gamma_d(X^i\cap X_\eta)\subseteq \Gamma_d(X^i, X_\eta)\subseteq X.$$
Here the first equality  follows directly from \eqref{GammaSch}. Hence, we are done.
\end{proof}

\subsection{Proof of Seidel representation}
 We consider the operators $\mathcal{H}$, $\mathcal{T}$ on $QK(X)$, defined respectively by
   $$\mathcal{H}(\mathcal{O}^\lambda)=\mathcal{O}^{n-k}*\mathcal{O}^\lambda,\quad\mbox{and}\quad \mathcal{T}(\mathcal{O}^\lambda)=\mathcal{O}^{1^k}*\mathcal{O}^\lambda.$$
Under the isomorphism $Gr(k, n)\cong Gr(n-k, n)$, Theorem \ref{SeidelQKT} is equivalent to Theorem \ref{SeidelQK} as follows.  As a consequence, we have $\mathcal{H} \mathcal{T}=q\mbox{Id},\, \mathcal{H}^n=q^{n-k}\mbox{Id}$ and $\mathcal{T}^n=q^{k}\mbox{Id}$. Moreover, $\mathcal{T}|_{q=1}$  generates an action of the cyclic group $\mathbb{Z}/n\mathbb{Z}$ on  $QK(X)|_{q=1}$, called the \textit{Seidel representation} on $QK(X)$.

\begin{thm}\label{SeidelQK} Let $\mu\in \mathcal{P}_{k, n}$. In $QK(X)$, we have
  \begin{equation}\label{eqnHH}
     \mathcal{H}(\mathcal{O}^\mu)= \mathcal{O}^{n-k}*\mathcal{O}^\mu=q^{n-k+|\mu|-|\mu\downarrow1|\over n}\mathcal{O}^{\mu\downarrow 1}=\begin{cases}
         q\mathcal{O}^{(\mu_1-1,\cdots, \mu_k-1)},&\mbox{if }\mu_k>0,\\
      \mathcal{O}^{(n-k, \mu_1, \cdots, \mu_{k-1})},&\mbox{if } \mu_k=0.
  \end{cases}
  \end{equation}
 \end{thm}
\begin{proof}
Assume $\mu_k>0$ first. Then $(n-k, 0,\cdots, 0)\not\leq \mu^\vee$ in the Bruhat order, so that $X^{n-k}\cap X_{\mu^\vee}=\emptyset$.
 Thus we have  $\mathcal{O}^{n-k}\cdot \mathcal{O}^\mu=0$ in $K(X)$.
For any $\nu\in \mathcal{P}_{k, n}$, 
   \begin{align*}
   N_{n-k, \mu}^{\nu, 1}
             =I_1(\mathcal{O}^{n-k}, \mathcal{O}^\mu, \xi_{\nu})&=\chi_X([\mathcal{O}_{ \Gamma_1(X^{n-k},  X_{\mu^\vee})}]\cdot \xi_\nu)\\&=\chi_X([\mathcal{O}_{ g\cdot X_{\mu^\vee \uparrow 1}}]\cdot \xi_\nu)=\chi_X(\mathcal{O}^{\mu  \downarrow 1}\cdot \xi_\nu)=\delta_{\mu\downarrow 1, \nu} .
         \end{align*}
 Notice $\mu^\vee\uparrow 1=(\mu\downarrow 1)^\vee$. Using  \eqref{sc2}, \eqref{KGW3}  and induction on $d\geq 2$, we have
 \begin{align*}
   N_{n-k, \mu}^{\nu, d}&= I_d(\mathcal{O}^{n-k}, \mathcal{O}^\mu, \xi_{\nu})- I_{d-1}(\mathcal{O}^{\mu\downarrow 1},\xi_\nu)\\
   &= I_d(\mathcal{O}^{n-k}, \mathcal{O}^\mu, \xi_{\nu})- I_{d-1}(\mathcal{O}^{\rm id}, \mathcal{O}^{\mu\downarrow 1},\xi_\nu)\\
   &=   \chi_X([\mathcal{O}_{\Gamma_{d}(X^{n-k}, X_{\mu^\vee})}]\cdot \xi_\nu)-\chi_X([\mathcal{O}_{\Gamma_{d-1}(X^{\mu\downarrow 1})}]\cdot \xi_\nu)\\
   &=   \chi_X([\mathcal{O}_{\Gamma_{d-1}(g\cdot X_{\mu^\vee\uparrow 1})}]\cdot \xi_\nu)-\chi_X([\mathcal{O}_{\Gamma_{d-1}(X^{\mu\downarrow 1})}]\cdot \xi_\nu)\\
   &=0.
    \end{align*}
Here we denote $\mathcal{O}^{\rm id}:=\mathcal{O}^{(0, \ldots, 0)}$, which is the identity element.   The third equality follows from  Corollary \ref{compareGamma}.

Now we assume $\mu_k=0$, and  notice $\mu^\vee_1=n-k$. Hence,
  \begin{align*}
      X^{n-k}\cap X_{\mu^\vee}
      &=\{V_{k}\mid \dim V_k\cap F_1\geq 1, \dim V_{k}\cap F^{\rm opp}_{j+\mu^\vee_{k-j+1}}\geq j, 1\leq j\leq k\}\\
                       &=\{V_{k}\mid  \dim V_k\cap F_1\geq 1, \dim V_{k}\cap (F_1+F^{\rm opp}_{j+\mu^\vee_{k-j+1}})\geq j+1, \\
                       &{} \qquad 1\leq j\leq k-1\}\\
                        &= X^{\mu\downarrow 1}(g^{-1}\cdot F^{\rm opp}_\bullet)=g^{-1}\cdot X^{\mu\downarrow 1}(F^{\rm opp}_\bullet),
  \end{align*}  
  where $g$ was given in (\ref{matrix g}). Hence,  $\mathcal{O}^{n-k}\cdot \mathcal{O}^\mu=\mathcal{O}^{\mu\downarrow 1}$. Moreover,   for any $\nu\in \mathcal{P}_{k, n}$, using \eqref{KGW3} we have
\begin{align*}
   N_{n-k, \mu}^{\nu, 1}
            &= I_1(\mathcal{O}^{n-k}, \mathcal{O}^\mu, \xi_{\nu})-\sum_{\kappa\in \mathcal{P}_{k, n}}N_{n-k, \mu}^{\kappa, 0}I_1(\mathcal{O}^\kappa,\xi_\nu)\\
            &=I_1(\mathcal{O}^{n-k}, \mathcal{O}^\mu, \xi_{\nu})- I_1(\mathcal{O}^{\rm id}, \mathcal{O}^{\mu\downarrow 1},\xi_\nu)\\
            &=\chi_X([\mathcal{O}_{\Gamma_1( X^{n-k}, X_{\mu^\vee})}]\cdot \xi_\nu)-\chi_X([\mathcal{O}_{\Gamma_1( X^{\mu\downarrow 1})}]\cdot \xi_\nu)\\
              &=\chi_X([\mathcal{O}_{\Gamma_1( X^{n-k}\cap X_{\mu^\vee})}]\cdot \xi_\nu)-\chi_X([\mathcal{O}_{\Gamma_1( X^{\mu\downarrow 1})}]\cdot \xi_\nu)\\
        &=\chi_X([\mathcal{O}_{\Gamma_1( g^{-1}\cdot X^{\mu \downarrow 1}(F^{\rm opp}_\bullet))}]\cdot \xi_\nu)-\chi_X([\mathcal{O}_{\Gamma_1( X^{\mu\downarrow 1})}]\cdot \xi_\nu)\\
        &=0.
 \end{align*}
 Here the fourth equality follows from Proposition \ref{Gammacomparemukzero}.  By induction on $d$ and using   Proposition \ref{Gammacomparemukzero}, we conclude  $N_{n-k, \mu}^{\nu, d}=0$ for all $d>0$.
 \end{proof}

\section{An alternative proof of quantum Pieri rule}

 For $\lambda,\nu \in \mathcal{P}_{k, n}$, by $\lambda\leq \nu$ in the Bruhat order we mean $\lambda_i\leq \nu_i$ for all $i$.  In this paper, by  $\nu/\lambda$ we always request    $\lambda\leq \nu$ such that the complement $\nu/\lambda$ is a horizontal strip (i.e. $\nu_{i+1}\leq \lambda_i$ for all $1\leq i\leq k-1$).
Denote by $r(\nu/\lambda)$ the number of nonempty rows in the skew diagram $\nu/\lambda$. Let $\ell(\lambda)$ represent the number of index $i$ with $\lambda_i \neq 0$ for all $1\leq i \leq k$.   The following quantum Pieri rule was proved by Buch and Mihalcea \cite[Theorem 5.4]{BuMi}. Its   degree zero part  is the classical Pieri rule due to Lenart \cite[Theorem 3.2]{Lena}.

\begin{prop}[Quantum Pieri rule]\label{propQuanPieri} Let  $\lambda \in \mathcal{P}_{k, n}$ and $1\leq i\leq n-k$. 
We have
  $$\mathcal{O}^\lambda*\mathcal{O}^{i}=\sum_{i\leq |\nu/\lambda|\leq i+ r(\nu/\lambda)-1} (-1)^{|\nu/\lambda|-i}{r(\nu/\lambda)-1\choose |\nu/\lambda|-i}\mathcal{O}^{\nu}+q\sum (-1)^{e}{\varrho\choose e}\mathcal{O}^{\nu}.$$
Here the second sum occurs only if      $\ell(\lambda)=k$ and $\nu$ can be obtained from $\lambda$ by removing a subset of the boxes in the outer rim of $\lambda$ with at least one box removed from each row. When this holds, $e= |\nu|+n-|\lambda|-i$,  and  $\varrho$ counts the number of rows of $\nu$ that contain at least one box from the outer rim of $\lambda$, excluding the bottom row of this rim.
\end{prop}
\begin{example}\label{examGr49}{
 Consider $N_{i, \lambda}^{\nu, 1}$ for $QK(Gr(4, 9))$, where $i=4$, $\lambda=(4,3,2,1)$ and $\nu=(3,2,1,0)$.
   The boxes inside the  outer rim of $\lambda$ are shaded below.


  We have   $e=|\nu|+9-|\lambda|-4=1$ and $\varrho=3$.  Thus  $N_{i, \lambda}^{\nu, 1}=-3$, giving the negative coefficient of the following product in $QK(Gr(4,9))$:
$$\mathcal{O}^4*\mathcal{O}^{(4,3,2,1)}= \mathcal{O}^{(5,4,3,2)}+q\mathcal{O}^{(2,2,1,0)}+q\mathcal{O}^{(3,1,1)}+q\mathcal{O}^{(3,2,0,0)}-3q\mathcal{O}^{(3,2,1,0)}.$$
Moreover, we let $\tilde \nu:=\nu\downarrow \lambda_4=(5,3,2,1)$ and $\tilde \lambda:=\lambda\downarrow \lambda_4=(3,2,1,0)$. Then
we have $r(\tilde \nu/\tilde \lambda)-1=3$ and $|\tilde \nu/\tilde \lambda|-i=1$, and hence $N_{i, \tilde \lambda}^{\tilde \nu, 0}=(-1)^1{3\choose 1}=-3=N_{i, \lambda}^{\nu, 1}$. }
\end{example}

\subsection{An alternative proof of Proposition \ref{propQuanPieri}}
We assume the following lemma first, and leave the proof in the next subsection.
\begin{lemma}\label{lemTwopointRicardson}
 Let $\eta\in \mathcal{P}_{k, n}$ and $1\leq i\leq n-k$. If $\eta_1=n-k$,  then we have 
     \begin{eqnarray}
    \label{KGW4}   I_d([\mathcal{O}_{X^i\cap X_\eta}], \gamma)&=&  \chi_X( [\mathcal{O}_{\Gamma_d(X^i\cap X_\eta)}]\cdot \gamma)\quad\mbox{for } 1\leq d<\min\{k, n-k\}.
\end{eqnarray}
\end{lemma}
\begin{lemma}\label{lemTwopointRicardsonlarged}
 Let $\lambda,\eta\in \mathcal{P}_{k, n}$. If $X^\lambda\cap X_\eta\neq \emptyset$,  then we have 
     \begin{eqnarray}
    \label{KGW4}   I_d([\mathcal{O}_{X^\lambda\cap X_\eta}], \gamma)&=&  \chi_X( [\mathcal{O}_{\Gamma_d(X^\lambda\cap X_\eta)}]\cdot \gamma)\quad\mbox{for } d\geq \min\{k, n-k\}.
\end{eqnarray}
\end{lemma}
\begin{proof}
   Since $d\geq \min\{k, n-k\}$,  we have $\Gamma_d(\mbox{point})=X$ by \eqref{GammaSch}, implying that $\Gamma_d(X^\lambda, X_\eta)=\Gamma_d(X^\lambda\cap X_\eta)=X$.
   Notice  $\chi_X([\mathcal{O}_Z])=1$, whenever $Z\subset X$ is a closed torus-invariant subvariety that is unirational
and has rational singularities \cite[Corollary 4.18]{Deba}. In particular, we have
$$1=\chi_X([\mathcal{O}_{X^\lambda\cap X_{\eta}}])=\chi_X(\mathcal{O}^\lambda\cdot \mathcal{O}^{\eta^\vee})=\chi_X(\sum\nolimits_{\kappa}N_{\lambda, \eta^\vee}^{\kappa, 0}\mathcal{O}^\kappa)=\sum\nolimits_{\kappa}N_{\lambda, \eta^\vee}^{\kappa, 0}.$$
\begin{align*}
   I_d([\mathcal{O}_{X^\lambda\cap X_\eta}], \gamma) =I_d(\mathcal{O}^\lambda\cdot \mathcal{O}^{\eta^\vee}, \gamma)
                    &=\sum_\kappa N_{\lambda,\eta^\vee}^{\kappa, 0} I_d(\mathcal{O}^\kappa, \gamma)\\
                    &=\sum_\kappa N_{\lambda,\eta^\vee}^{\kappa, 0} \chi_X([\mathcal{O}_{\Gamma_d(X^\kappa)}], \gamma)\\
                   &=\sum_\kappa N_{\lambda,\eta^\vee}^{\kappa, 0} \chi_X([\mathcal{O}_{X}], \gamma)\\
                      &= \chi_X([\mathcal{O}_{\Gamma_d(X^\lambda\cap X_\eta)}], \gamma).
\end{align*}

\end{proof}
\begin{cor}\label{Pierideg0}
Let $\mu\in \mathcal{P}_{k, n}$ and $1\leq i\leq n-k$. If $\mu_k=0$,  then we have
$$\mathcal{O}^{i}*\mathcal{O}^\mu=\mathcal{O}^{i}\cdot \mathcal{O}^\mu.$$
\end{cor}

\begin{proof} Since $\mu_k=0$, $\mu^\vee_1=n-k$.
  For any $\nu\in \mathcal{P}_{k, n}$,
     by \eqref{sc2} we have
\begin{align*}
   N_{i, \mu}^{\nu, 1}
            &= I_1(\mathcal{O}^{i}, \mathcal{O}^\mu, \xi_{\nu})-\sum_{\kappa\in \mathcal{P}_{k, n}}N_{i, \mu}^{\kappa, 0}I_1(\mathcal{O}^\kappa,\xi_\nu)\\
            &=I_1(\mathcal{O}^{i}, \mathcal{O}^\mu, \xi_{\nu})- I_1(\mathcal{O}^{i}\cdot \mathcal{O}^\mu,\xi_\nu)\\
            &=I_1(\mathcal{O}^{i}, \mathcal{O}^\mu, \xi_{\nu})- I_1([\mathcal{O}_{X^{i}\cap X_{\mu^\vee}}],\xi_\nu)\\
            &=\chi_X([\mathcal{O}_{\Gamma_1( X^{i}, X_{\mu^\vee})}]\cdot \xi_\nu)-\chi_X([\mathcal{O}_{\Gamma_1( X^{i}\cap X_{\mu^\vee})}]\cdot \xi_\nu)\\
           &=0
 \end{align*}
Here the third equality follows from Lemma \ref{lemTwopointRicardson}, and the fourth equality follows from Proposition \ref{Gammacomparemukzero}.   Thus $N_{n-k, \mu}^{\nu, d}=0$ for all $1\leq d<\min\{k, n-k\}$, by using Lemma \ref{lemTwopointRicardson}, Proposition \ref{Gammacomparemukzero} and induction on $d$.

As a consequence, for $d=\min\{k, n-k\}$, by Lemma \ref{lemTwopointRicardsonlarged} we have
 \begin{align*}
 N_{i, \mu}^{\nu, d}
            &= I_d(\mathcal{O}^{i}, \mathcal{O}^\mu, \xi_{\nu})- I_d(\mathcal{O}^i\cdot \mathcal{O}^\mu,\xi_\nu) \\
            &=\chi_X([\mathcal{O}_{\Gamma_d( X^{i}, X_{\mu^\vee})}]\cdot \xi_\nu)-\chi_X([\mathcal{O}_{\Gamma_d( X^{i}\cap X_{\mu^\vee})}]\cdot \xi_\nu)\\
             &=0.
 \end{align*}
The last equality holds by noting   $\Gamma_d( X^{i}, X_{\mu^\vee})=\Gamma_d( X^{i}\cap X_{\mu^\vee})=X$. Hence,
  $N_{i, \mu}^{\nu, d}=0$ for all $d\geq \min\{k, n-k\}$, by using Lemma \ref{lemTwopointRicardsonlarged} and induction on $d$.
\end{proof}

\begin{remark}
  The above corollary is also a consequence of \cite[Proposition 7.1]{BCMP22}.
\end{remark}
 We remind of our assumption for    $\nu/\lambda$ that   $\lambda\leq \nu$ with the complement  of $\lambda$ in $\nu$ being a horizontal strip.
Now we can reprove 
Proposition \ref{propQuanPieri}  in the following form.
\begin{prop} \label{propQuanPieri22} For any  $\lambda \in \mathcal{P}_{k, n}$ and $1\leq i\leq n-k$, in $QK(Gr(k, n))$ we have
  $$\mathcal{O}^\lambda*\mathcal{O}^{i}=\sum (-1)^{|\nu/\lambda|-i}{r(\nu/\lambda)-1\choose |\nu/\lambda|-i}\mathcal{O}^{\nu}+q\sum (-1)^{|\nu|+n-i-|\lambda|}{r(\tilde \nu/\tilde \lambda)-1\choose |\tilde \nu/\tilde \lambda|-i}\mathcal{O}^{\nu}.$$
Here $i\leq |\nu/\lambda|\leq i+ r(\nu/\lambda)-1$ in the first sum.
Set
 $\tilde \lambda=\lambda\downarrow \lambda_k$. 
The second sum occurs only if $\lambda_k>0$, and when this holds it is over partitions $\nu$ such that the associated partition $\tilde \nu$ defined by  $$\tilde \nu_1=\nu_k-\lambda_k+n-k+1, \tilde \nu_i= \nu_{i-1}-\lambda_k+1, \, 2\leq i\leq k$$
 satisfies    with {\upshape a)} 
    $i\leq |\tilde \nu/\tilde \lambda|\leq i+r(\tilde \nu/\tilde \lambda)-1$ and {\upshape b)} 
 $\tilde{\nu}_1>n-k-\lambda_k$.
\end{prop}
\begin{proof} By Corollary \ref{Pierideg0}, $N_{\lambda, i}^{\nu, d}\neq 0$ for   $d>0$, only if $\lambda_k>0$.
Noting  $\tilde \lambda_k=0$, we have
    $\mathcal{O}^\lambda=\mathcal{T}^{\lambda_k}(\mathcal{O}^{\tilde \lambda})$. By Corollary \ref{Pierideg0} and the associativity in $QK(X)$, we have
    $$\mathcal{O}^\lambda*\mathcal{O}^{i}=\mathcal{T}^{\lambda_k}(\mathcal{O}^{\tilde \lambda})*\mathcal{O}^i=\mathcal{T}^{\lambda_k}(\mathcal{O}^{\tilde \lambda}*\mathcal{O}^i)=\mathcal{T}^{\lambda_k}(\mathcal{O}^{\tilde \lambda} \cdot \mathcal{O}^i)=\mathcal{T}^{\lambda_k}(\sum\nolimits_{\tilde \nu} N_{\tilde \lambda, i}^{\tilde \nu, 0}\mathcal{O}^{\tilde \nu}).$$
    Hence, $\mathcal{O}^\lambda*\mathcal{O}^{i}=\sum_{\tilde \nu}q^{d_{\lambda_k}}N_{\tilde \lambda, i}^{\tilde \nu, 0} \mathcal{O}^{\tilde \nu\uparrow \lambda_k}$, where $d_{\lambda_k}:=d_{\lambda_k}(I_{\tilde \nu})$ was defined in \eqref{SeidelQHdlambdak}.
    In other words, for $\nu=\tilde \nu\uparrow \lambda_k$, we have $\tilde \nu=\nu\downarrow \lambda_k$ and
     $N_{\lambda, i}^{\nu, d_{\lambda_k}}=N_{\tilde \lambda, i}^{\tilde \nu, 0}$.
   By Lenart's Pieri rule,  $N_{\tilde \lambda, i}^{\tilde \nu, 0}\neq 0$ only if $\tilde \nu/\tilde \lambda$ is a horizontal strip with
    $|\tilde \nu/\tilde \lambda|\geq i$, which implies $\tilde \nu_2\leq \tilde \lambda_1=\lambda_1-\lambda_k$.
   By Theorem \ref{SeidelQKT}, we can write $\mathcal{T}^j(\mathcal{O}^{\tilde \nu})=q^{d_j}\mathcal{O}^{\tilde \nu\uparrow j}$. The power $d_j$ is increasing in $j$; for $j=n-k-\tilde \nu_2+1$, we notice $d_j=2$ and $d_{j-1}=1$. Since
    $$n-k-\tilde \nu_2+1\geq n-k-(\lambda_1-\lambda_k)+1=(n-k-\lambda_1)+1+\lambda_k>\lambda_k,$$
    it follows    that $d_{\lambda_k}\leq 1$. Moreover,   $d_k\geq 1$ holds if and only if
       $\lambda_k>n-k-\tilde \nu_1=:r$, by noticing  $\mathcal{T}^{r}(\mathcal{O}^{\tilde \nu})=\mathcal{O}^{(n-k, \tilde \nu_2+r, \cdots, \tilde \nu_k+r)}$. (We refer to section 6 for further study of the reductions of quantum-to-classical types by using $\mathcal{T}$.)

       In a summary, we have  $N_{\lambda, i}^{\nu, d}=0$ for any $d>1$. Moreover, $N_{\lambda, i}^{\nu, 1}\neq 0$
        if and only if $N_{\tilde \lambda, i}^{\tilde \nu, 0}\neq 0$,  $\lambda_k>0$ and   $\lambda_k>n-k-\tilde \nu_1$ all hold.
      When all these hold, we have
       $\nu=\tilde \nu\uparrow \lambda_k=(\tilde \nu_2+\lambda_k-1, \cdots, \tilde \nu_k+\lambda_k-1, \lambda_k-n+k+\tilde \nu_1-1)$.
\end{proof}

\begin{remark}{\upshape It follows   from the expression of $\nu$ that  $|\tilde \nu/\tilde \lambda|=|\nu|+n-|\lambda|$.
Moreover, we have $\lambda_j>\nu_j\geq \lambda_{j+1}-1$ for all $j$, where $\lambda_{k+1}:=0$. This   shows that $\nu$ is obtained from $\lambda$ by removing a subset of the boxes in the outer rim of $\lambda$ with at least one box removed from each row.
For $2\leq j\leq k$, the $j$-th row of $\tilde \nu/\tilde \lambda$ makes contributions in the counting $r(\tilde \nu/\tilde \lambda)$
   if and only if $\tilde \nu_j>\tilde \lambda_j$, equivalently $\nu_{j-1}=\tilde \nu_j+\lambda_k-1>\lambda_j-1$;  namely the $(j-1)$-th row of $\nu$ contain at least one box from the outer rim of $\lambda$.
   Since $\tilde \nu_1>n-k-\lambda_k\geq \lambda_1-\lambda_k=\tilde \lambda_1$, the first row of $\tilde \nu/\tilde \lambda$ makes contribution in the counting.
   Hence $r(\tilde \nu/\tilde \lambda)-1$ coincides with  the number of rows of $\nu$ that contain at least one box from the outer rim of $\lambda$, excluding the bottom row of this rim. In a summary, the description of the second sum
    in the above proposition is indeed equivalent to that in Proposition \ref{propQuanPieri}.
    }
\end{remark}

\subsection{Proof of Lemma \ref{lemTwopointRicardson}}

\subsubsection{Basic facts on cohomologicality} We start with
  some facts about cohomologically trivial morphisms, as were collected in \cite[section 2]{BCMP11}.

\begin{defn}
   {\upshape A morphism $f : Y\to Z$  of schemes is called \textit{cohomologically trivial}, if
 $f_*\mathcal{O}_Y=\mathcal{O}_Z$ and $R^if_*\mathcal{O}_Y = 0$ for $i > 0$.
 }
\end{defn}
\begin{defn} {\upshape
   An irreducible complex variety $Y$ has \textit{rational singularities}, if there exists a
cohomologically trivial resolution $\varphi:\tilde Y\to Y$, namely $\tilde Y$ is a nonsingular
variety and $\varphi$ is a proper birational and cohomologically trivial morphism.
}
\end{defn}
The following is  \cite[Proposition 2.2]{BCMP11},  proved in \cite[Theorem 3.1]{BuMi} by Buch and Mihalcea  as an application of   \cite[Theorem 7.1]{Koll} by Koll\'ar.
\begin{prop}\label{propcohtriv}
  Let $f : Y\to Z$ be a surjective morphism between complex irreducible
projective varieties with rational singularities. Assume that the general fibers of $f$
are rationally connected. Then $f$ is cohomologically trivial.
\end{prop}

\begin{lemma}[Lemma 2.4 of \cite{BCMP11}]\label{lemcompcohtri}
   Let $f_1 : Y_1 \to Y_2$ and $f_2: Y_2\to Y_3$ be morphisms of schemes. Assume $f_1$ to be  cohomologically trivial. Then $f_2$ is cohomologically trivial if and only
if $f_2\circ f_1$ is cohomologically trivial.
\end{lemma}
The next property is a useful criterion for rational connectedness, which was  conjectured by Koll\'ar, Miyaoka and Mori, and was proved  by Graber, Harris and Starr.
\begin{prop}[Corollary 3 of \cite{GHS}]\label{propratconn} Let $f : Y \to Z$ be any dominant morphism
of complete irreducible complex varieties. If $Z$ and the general fibers of $f$ are
rationally connected, then $Y$ is rationally connected.
\end{prop}

 We also need the following property about projected Richardson varieties proved in \cite{BiCo, KLS}. The projections $\pi_G, \pi_F, pr_1, pr_2$ in diagram \eqref{diag}
are the natural projections among flag varieties $G/P$ when $G=SL(n, \mathbb{C})$. Moveover, $F\ell_n:=F\ell_{1, 2, \cdots, n-1; n}$ is the special case of complete flag variety  $G/B$
when $G=SL(n, \mathbb{C})$. Let $\rho: G/B\to G/P$ denote the natural projection.
\begin{prop}\label{projRVcoh}
    Let $R \subset G/B$ be a Richardson variety.
\begin{enumerate}
  \item  The projected Richardson variety  $\rho(R)\subset G/P$ 
   has   rational singularities.
 \item  The restricted map $\rho: R \to  \rho(R)$ is cohomologically trivial.
\end{enumerate}
\end{prop}
\subsubsection{Cohomological triviality of $\pi_G|_{Z_d(X^i\cap X_\eta)}$}
Consider the fiber product $$Z_d\times_{F\ell_{k-d, k+d; n}} Z_d=\{(V_{k-d}, V_k, \bar V_k, V_{k+d})\mid V_{k-d}\leqslant V_k\leqslant V_{k+d}, V_{k-d}\leqslant \bar V_k\leqslant V_{k+d}\}.$$
 Let $\pi_3:  Z_d\times_{F\ell_{k-d, k+d; n}} Z_d \to
F\ell_{k-d, k, k+d; n}$ and $\pi_4:Z_d\times_{F\ell_{k-d, k+d; n}} Z_d \to Gr(k, n)$ be the natural projections defined by mapping
  $(V_{k-d}, V_k, \bar V_k, V_{k+d})$ to $V_{k-d}\leqslant V_k\leqslant V_{k+d}$ and $\bar V_k$ respectively.  Since $\pi_4$ is smooth and the fiber of $\pi_4$ is connected, it follows from the properties of the Richardson variety $X^i\cap X_\eta\subset Gr(k, n)$ that the incidence variety
    $$IV:=\pi_4^{-1}(X^i\cap X_\eta)$$ is irreducible and has rational singularities.
     Recall $$Z_d(X^i\cap X_\eta)=\{V_{k-d}\leqslant V_k\leqslant V_{k+d}\mid \exists \bar V_k\in X^i\cap X_\eta \mbox{ with } V_{k-d}\leqslant \bar V_k\leqslant V_{k+d}\}.$$ Therefore we have $\pi_3(IV)=Z_d(X^i\cap X_\eta)$.  For $x\in pr_1(Z_d(X^i\cap X_\eta))$,
  we set $F_x:= \pi_3|_{IV}^{-1}\big(pr_1|_{Z_d(X^i\cap X_\eta)}^{-1}(x)\big)\subset IV$, and   consider the following surjective morphisms, where we still denote by $\pi_i$ the restriction maps by abuse of notation.

    \begin{equation*}
    \scalebox{0.85}{
   \xymatrix{
  x\in pr_1(Z_d(X^i\cap X_\eta))    & IV   \ar[r]^{\pi_4 } \ar[l]_{{}\qquad\qquad pr_1\circ\pi_3} \ar[d]^{ \pi_3}  &  X^i\cap X_\eta    \\
  & Z_d(X^i\cap X_\eta)  \ar[0,0];[-1,-1]^{pr_1} &    }
 \xymatrix{
    & F_x   \ar[r]^{\pi_4  } \ar[d]^{ \pi_3  }  &   \pi_4(F_x)    \\
  &  pr_1|_{Z_d(X^i\cap X_\eta)}^{-1}(x)  &   }
  }
\end{equation*}
  Since $\eta_1=n-k$, $F^{\rm opp}_{k+\eta_{k-k+1}}=\mathbb{C}^n$. Thus   $\bar V_k\cap F^{\rm opp}_{k+\eta_{k-k+1}}\cap V_{k+d}\geq k$ holds if and only if  $\bar V_k \leqslant V_{k+d}$ holds. Thus for $x$ being  $V_k\leqslant V_{k+d}$,  we have
   \begin{align*}
      \pi_4(F_x)&=\{\bar V_k \leqslant \mathbb{C}^n \mid \dim \bar V_k \cap F_{n-k+1-i}\geq 1; \bar V_k\cap F^{\rm opp}_{s+\eta_{k-s+1}}\geq s, 1\leq s\leq k; \bar V_k\leqslant V_{k+d}\}\\
      &=\{\bar V_k \leqslant V_{k+d} \mid \dim \bar V_k \cap F_{n-k+1-i}\cap V_{k+d}\geq 1; \bar V_k\cap F^{\rm opp}_{s+\eta_{k-s+1}}\cap V_{k+d}\geq s, 1\leq s\leq k\}\\
      &=X^{\hat m}(F_\bullet \cap V_{k+d})\cap  X_{\hat \eta}(F^{\rm opp}_\bullet\cap V_{k+d}),
   \end{align*}
  for some partitions $(\hat m, 0,\cdots, 0)$ and $\hat \eta$ in $\mathcal{P}_{k, k+d}$.
The third equality holds, by noting that the intersections $F_\bullet\cap V_{k+d}$ and $F^{\rm opp}\cap V_{k+d}$ are induced complete flags in $V_{k+d}$. The complete flag $F_\bullet \cap V_{k+d}=h\cdot (F^{\rm opp}_\bullet \cap V_{k+d})$, where  $h\in SL(k, k+d)$, is no longer opposite
to  $F^{\rm opp}_\bullet \cap V_{k+d}$ in general.
  \begin{lemma}\label{lemmconnected}
     For any $x\in pr_1(Z_d(X^i\cap X_\eta))$, $\pi_4(F_x)$ is connected.
  \end{lemma}
\begin{proof}
  It follows from  $\eta_1=n-k$   that $\hat \eta_1=d$. (Namely   the dimension condition for $s=k$ is always redundant.)
Thus $(\hat m, 0,\cdots, 0)\leq (\hat\eta_1, \cdots, \hat\eta_k)$ in the Bruhat order, and consequently
  the Richardson variety  $X_{\hat \eta}(F^{\rm opp}_\bullet\cap V_{k+d})\cap   X^{\hat m}((F^{\rm opp}_\bullet \cap V_{k+d})^{\rm opp})$ in $Gr(k, V_{k+d})$ is nonempty. Hence, the statement follows by \cite[Proposition 3.2]{BCMP11}.
\end{proof}
     \begin{lemma}\label{lemirreducible}
     For generic  $x\in pr_1(Z_d(X^i\cap X_\eta))$, both $\pi_4(F_x)$ and $F_x$ are reduced, irreducible projective varieties.
  \end{lemma}
  \begin{proof}
   The fiber of $\pi_4|_{F_x}$ at  an arbitrary point $\bar V_k\in \pi_4(F_x)$ is given by
     $$\pi_4|_{F_x}^{-1}(\bar V_k)=\{V_{k-d}\mid V_{k-d}\leqslant \bar V_k, V_{k-d}\leqslant V_k\}=Gr(k-d, V_k\cap \bar V_k),$$
     and hence is connected. By Lemma \ref{lemmconnected}, $\pi_4(F_x)$ is connected, so is  $F_x$.

By Lemma \ref{lemmiddle}, $pr_1(Z_d(X^i\cap X_\eta))=pr_1(Z_d(X^i, X_\eta))$  is a projected Richardson variety in $F\ell_{k, k+d;n}$, since   $Z_d(X^i, X_\eta)$ is a Richardson variety in $F\ell_{k-d, k, k+d; n}$. In particular,  $pr_1(Z_d(X^i\cap X_\eta))$ is  irreducible. Since $IV$ is  irreducible and has rational singularities, the generic fiber $F_x$
 of $pr_1\circ \pi_3|_{IV}$ has rational singularities by \cite[Lemma 3]{Brio}. In particualr, $F_x$ is reduced and normal. Since $F_x$ is connected and normal,
 it follows that $F_x$ is irreducible.
 Hence, $\pi_4(F_x)$ is irreducible as well. Since $F_x$ is reduced, its scheme theoretic image $\pi_4(F_x)$ is reduced as well.
  \end{proof}

Every Schubert variety in a flag variety has a stratification by Schubert cells, with each Schubert cell isomorphic to an affine space and the largest Schubert cell being a Zariski open subset. In particular, Schubert varieties are rational. In \cite{SSV}, B. Shapiro, M. Shapiro and A. Vainshtein made
refined double decomposition for intersections of Schubert cells in $F\ell_n$, and studied topological properties of such intersections.
We will  just need  formal descriptions of   the stratums as follows,
and refer to \cite{SSV} for relevant notions and more precise descriptions. We remark that a refinement of Bruhat decomposition for   complete flag varieties of general Lie type was given by Curtis \cite{Curt}.

\begin{prop}[Theorem A. of \cite{SSV}]\label{refineddecomp} Let $F^{(1)}_\bullet$ and $F^{(2)}_\bullet$ be any two complete flags in $\mathbb{C}^N$.  Let $u, v$ be  permutations in $S_N$.
 The intersection of Schubert cells  $\overset{\circ}{X}_u(F^{(1)}_\bullet)$ and $\overset{\circ}{X}_v(F^{(2)}_\bullet)$ in $F\ell_n$ admits a refined double decomposition
     $$\overset{\circ}{X}_u(F^{(1)}_\bullet)\bigcap \overset{\circ}{X}_v(F^{(2)}_\bullet)=\bigsqcup_{U\in RD_{F^{(1)}_\bullet, F^{(2)}_\bullet}} U,  $$
with each stratum being biholomorphically equivalent to $(\mathbb{C}^*)^a\times \mathbb{C}^b$ for some $(a, b)$.
\end{prop}

\begin{cor}\label{corratconn}
    For generic  $x\in pr_1(Z_d(X^i\cap X_\eta))$,   $\pi_4(F_x)$ is rationally connected.
\end{cor}
\begin{proof}
  For the natural projection $\rho: F\ell_{k+d}\to Gr(k, V_{k+d})$, the preimage
 $\rho^{-1}(\pi_4(F_x))$ is an intersection of Schubert varieties $X^u(F_\bullet \cap V_{k+d})$ and $X_v(F^{\rm opp}_\bullet\cap V_{k+d})$ in $F\ell_{k+d}$ for some permutations $u, v$. Since $\pi_4(F_x)$ is reduced and irreducible by Lemma \ref{lemirreducible}, so is  $\rho^{-1}(\pi_4(F_x))$.
 Each Schubert variety has a stratification by Schubert cells. Thus
  $\rho^{-1}(\pi_4(F_x))$ is the disjoint union of intersection of Schubert cells in $F\ell_n$, and hence is the disjoint union of refined double stratums of the form  $(\mathbb{C}^*)^a\times \mathbb{C}^b$
   by Proposition \ref{refineddecomp}. Refined double stratums of the largest dimension are Zariski open subsets of  $\rho^{-1}(\pi_4(F_x))$. Since
    $\rho^{-1}(\pi_4(F_x))$ is reduced and irreducible, there exists a unique stratum of the largest dimension, say  $(\mathbb{C}^*)^a\times \mathbb{C}^b$, which is also reduced scheme-theoretically. Hence,   $\rho^{-1}(\pi_4(F_x))$ is birational to $\mathbb{P}^{a+b}$. Hence, $\rho^{-1}(\pi_4(F_x))$ is rationally connected, and consequently
    the statement follows.
\end{proof}
\begin{prop}\label{propcohtrivZd} Let $\eta\in \mathcal{P}_{k, n}$, $1\leq i\leq n-k$ and $1\leq d<\min\{k, n-k\}$. If $\eta_1=n-k$, then
   $\pi_G:  Z_d(X^i\cap X_\eta)\to \Gamma_d(X^i\cap X_\eta)$ is cohomologically trivial.
 \end{prop}
\begin{proof}
 Notice that $Z_d(X^i, X_\eta)$ is Richardon variety in $F\ell_{k-d, k, k+d; n}$. Consider the natural projection $\hat\rho: F\ell_{n}\to  F\ell_{k-d, k, k+d; n}$. Then    $Y:=\hat\rho^{-1}(Z_d(X^i, X_\eta))$ is again a Richardson variety. Hence, $\hat\rho|_Y$, $pr_1\circ \hat\rho|Y$ and  $\pi_G\circ \hat\rho|Y$ are all cohomologically trivial by Proposition \ref{projRVcoh} (2). Therefore $pr_2: pr_1(Z_d(X^i, X_\eta))=pr_1\circ \hat\rho(Y)\to \Gamma_d(X^i, X_\eta)$ is cohomologically trivial by Lemma \ref{lemcompcohtri}.
 Therefore by Lemma \ref{lemmiddle} and Proposition \ref{Gammacomparemukzero}, $pr_2: pr_1(Z_d(X^i\cap X_\eta)) \to \Gamma_d(X^i\cap X_\eta)$ is cohomologically trivial.

The variety   $\pi_F\pi_G^{-1}(X^i\cap X_\eta)$ is a projected Richardson variety in $F\ell_{k-d, k+d; n}$. Thus it is irreducible and    has rational singularities. Since $\pi_F$ is a smooth morphism, $Z_d(X^i\cap X_\eta)=\pi_F^{-1}\pi_F\pi_G^{-1}(X^i\cap X_\eta)$ is irreducible and  has rational singularities as well.
By Lemma \ref{lemmiddle}, $pr_1(Z_d(X^i\cap X_\eta))=pr_1(Z_d(X^i, X_\eta))$, and hence it is a projected Richardson variety. Consequently
 $pr_1(Z_d(X^i\cap X_\eta))$ is irreducible and has rational singularities. For generic $x\in   pr_1(Z_d(X^i\cap X_\eta))$,
  both $F_x$ and $\pi_4(F_x)$ are irreducible by Lemma \ref{lemirreducible}. The base $\pi_4(F_x)$ is rationally connected by  Corollary \ref{corratconn}. The fiber of $\pi_4|_{F_x}$ is   rationally connected, for being $Gr(k, V_k\cap \bar V_k)$. Therefore $F_x$ is rationally connected by Proposition \ref{propratconn}. Hence, the generic fiber $pr_1^{-1}(x)$ of $pr_1|_{Z_d(X^i\cap X_\eta)}$  is rationally connected.
Hence, $pr_1|_{Z_d(X^i\cap X_\eta)}$ is cohomologically trivial by Proposition \ref{propcohtriv}. Therefore
 $\pi_G|_ {Z_d(X^i\cap X_\eta)}=pr_2\circ pr_1|_{Z_d(X^i\cap X_\eta)}$ is cohomologically trivial by Lemma \ref{lemcompcohtri}.
\end{proof}

\subsubsection{Proof Lemma \ref{lemTwopointRicardson}} We state the non-equivariant version of     \cite[Theorem 4.2]{BuMi} as follows, where  $a=\max\{k-d, 0\}$ and $b=\min\{k+d, n\}$.

\begin{prop}\label{propQtoC}
   For  any classes $\alpha_1, \alpha_2, \alpha_3\in  K(X)$ and any $d\geq 1$, we have
    $$I_d(\alpha_1, \alpha_2, \alpha_3)=\chi_{F\ell_{a, b; n}}(\pi_F{}_*\pi^*_G(\alpha_1)\cdot \pi_F{}_*\pi^*_G(\alpha_2)\cdot \pi_F{}_*\pi^*_G(\alpha_2)). $$
\end{prop}

\begin{cor}\label{cortwoptRichardson}For any $\lambda, \eta\in \mathcal{P}_{k, n}$, $\gamma\in K(X)$ and $d\geq 1$, we have
     \begin{eqnarray}
    \label{KGW2}   I_d([\mathcal{O}_{X^\lambda\cap X_\eta}], \gamma)&=&  \chi_X(\pi_G{}_*[\mathcal{O}_{Z_d(X^\lambda\cap X_\eta)}]\cdot \gamma).
\end{eqnarray}
\end{cor}

\begin{proof} By Proposition \ref{propQtoC} and the projection formula, we have
\begin{align*}
 I_d([\mathcal{O}_{X^\lambda\cap X_\eta}], \gamma)&=\chi_{F\ell_{a, b; n}}(\pi_F{}_*\pi^*_G(\mathcal{O}^{\rm id})\cdot \pi_F{}_*\pi^*_G([\mathcal{O}_{X^\lambda\cap X_\eta}])\cdot \pi_F{}_*\pi^*_G(\gamma))\\
 &=\chi_{F\ell_{a, b; n}}( \pi_F{}_*\pi^*_G([\mathcal{O}_{X^\lambda\cap X_\eta}])\cdot \pi_F{}_*\pi^*_G(\gamma))\\
 &=\chi_{F\ell_{a,k, b; n}}(\pi_F^*\pi_F{}_*\pi^*_G([\mathcal{O}_{X^\lambda\cap X_\eta}])\cdot \pi^*_G(\gamma))\\
 &=\chi_{X}(\pi_G{}_*\pi_F^*\pi_F{}_*\pi^*_G([\mathcal{O}_{X^\lambda\cap X_\eta}])\cdot  \gamma).
\end{align*}
Since $\pi^{-1}_G(X^\lambda\cap X_\eta)$ is a Richardson variety, $\pi_F|_{\pi^{-1}_G(X^\lambda\cap X_\eta)}$ is cohomologically trivial by Proposition \ref{projRVcoh} (2) (more precisely, by the same arguments as at the beginning of Proposition \ref{propcohtrivZd}). It follows that $\pi_F^*\pi_F{}_*\pi^*_G([\mathcal{O}_{X^\lambda\cap X_\eta}])=[\mathcal{O}_{\pi_F^{-1}\pi_F\pi^{-1}_G(X^\lambda\cap X_\eta)}]=\mathcal{O}_{Z_d(X^\lambda\cap X_\mu)}$. Therefore the statement follows.
\end{proof}

\bigskip
\begin{proof}[Proof of Lemma \ref{lemTwopointRicardson}]
  We have  $I_d([\mathcal{O}_{X^i\cap X_\eta}], \gamma) =   \chi_X(\pi_G{}_*[\mathcal{O}_{Z_d(X^i\cap X_\eta)}]\cdot \gamma)$ by Corollary \ref{cortwoptRichardson}.
    By   Proposition \ref{propcohtrivZd}, $\pi_G: Z_d(X^i\cap X_\eta)\to \Gamma_d(X^i\cap X_\eta)$ is cohomologically trivial. Since the projection $\pi_G:{Z_d(X^i\cap X_\eta)}\to \Gamma_d(X^i\cap X_\eta)$ is proper and surjective, we have $\pi_G{}_*[\mathcal{O}_{Z_d(X^i\cap X_\eta)}]=[\mathcal{O}_{\Gamma_d(X^i\cap X_\eta)}]$. Thus the statement follows.
\end{proof}
 \section{Quantum-to-classical for certain structure constants}
As in \cite[Conjecture 5.10]{BuMi}, the structure constants for $QK(X)$ are expected to satisfy the alternating positivity: $(-1)^{|\lambda|+|\mu|+|\nu|+dn}N_{\lambda, \mu}^{\nu, d}\geq 0$.  {This was recently proved in \cite{BCMP22} for minuscule Grassmannians and quadric hypersurfaces with a geometric method.}
It is then very natural to ask for a quantum Littlewood-Richardson rule for all $N_{\lambda, \mu}^{\nu, d}$, which is a central theme in the subject of Schubert calculus.
The classical  Littlewood-Richardson rule for all $N_{\lambda, \mu}^{\nu, 0}$ was first given by Buch \cite{BuchGr}. In this section, we will prove Theorem \ref{mainthm1}, which  ensures that the structure constants $N_{\lambda, \mu}^{\nu, d_{\rm min}}$ for the smallest power $q^{d_{\rm min}}$     appearing in $\mathcal{O}^\lambda*\mathcal{O}^\mu$ are all equal to corresponding  classical Littlewood-Richardson coefficients.
 We will also reduce a bit more structure constants $N_{\lambda, \mu}^{\nu, d}$   to structure constants of smaller degree. 
Similar   properties 
have been studied for $QH^*(X)$ by Postnikov \cite[Proposition 6.10]{Post}.\footnote{Postnikov \cite[Proposition 6.10]{Post} also did the quantum-to-classical reduction for the largest power of $q$ appearing in a quantum product, while this part cannot be generalized to $QK(X)$, because of the lack of strange duality.}
The sufficient condition we provide looks more accessible.

\subsection{Proof of Theorem \ref{mainthm1}}
For convenience, we restate  Theorem \ref{mainthm1} as follows.
 It is   the quantum $K$-version of a formula  by Belkale in the proof of \cite[Theorem 10]{Belk},  {or equivalently the quantum $K$-version of
   \cite[Corollary 6.2 and the $D_{\rm min}$-part of Theorem 7.1]{Post} by Postnikov.}

 \begin{thm}\label{mainthm2}
     Let $\lambda, \mu \in \mathcal{P}_{k, n}$.
 The smallest power $d_{\rm min}$ of $q$ appearing in  $\mathcal{O}^\lambda*\mathcal{O}^\mu$ in $QK(X)$ equals that appearing in  $[X^\lambda]\star [X^\mu]$ in $QH^*(X)$, and   is given by 
$$d_{\rm min} = \max\{{1\over n}\big({|\lambda|-|\lambda\uparrow i|+|\mu|-|\mu\uparrow (n-i)|} \big)| 0\leq i\leq  n\}.$$ Moreover, if the max
is achieved for $r$, then
   $$\mathcal{O}^\lambda* \mathcal{O}^\mu=q^{d_{\rm min}} \mathcal{O}^{\lambda\uparrow r}*\mathcal{O}^{\mu\uparrow (n-r)}. $$
  \end{thm}

\begin{proof}  Notice that we have obtained  \eqref{eqnTT}, due to  \eqref{eqnHH} and  the isomorphism $Gr(k, n)\cong Gr(n-k, n)$. Then for any $i\geq 0$, 
it follows from \eqref{SeidelQH}  that
$\mathcal{T}^i(\mathcal{O}^\lambda)=q^{a}\mathcal{O}^{\hat \lambda}\in QK(X)$ if and only if  $ {T}^i([X^\lambda])=q^{a}[X^{\hat \lambda}] \in QH^*(X)$. Therefore 
 \begin{align*}
   q^k\mathcal{O}^\lambda*\mathcal{O}^\mu&=\mathcal{T}^n(\mathcal{O}^\lambda*\mathcal{O}^\mu)\\
        &=\mathcal{T}^i(\mathcal{O}^\lambda)*\mathcal{T}^{n-i}(\mathcal{O}^\mu)\\
         &=q^{{1\over n}\big(ik+|\lambda|-|\lambda\uparrow i|\big)}\mathcal{O}^{\lambda\uparrow i}*q^{{1\over n}\big((n-i)k+|\mu|-|\mu\uparrow (n-i)|\big)}\mathcal{O}^{\mu\uparrow (n-i)}.
\end{align*}
\noindent
Here the second equality follows from the associativity and commutativity of the quantum $K$ product among Schubert classes.  It follows that $$\mathcal{O}^\lambda*\mathcal{O}^\mu=q^{{1\over n}\big({|\lambda|-|\lambda\uparrow i|+|\mu|-|\mu\uparrow (n-i)|} \big)}\mathcal{O}^{\lambda\uparrow i}*\mathcal{O}^{\mu\uparrow (n-i)}$$ for all $0\leq i\leq n$.
In particular for $d_{\rm min}=\max\{ {1\over n}\big({|\lambda|-|\lambda\uparrow i|+|\mu|-|\mu\uparrow (n-i)|} \big)\mid 0\leq i\leq n\}$ which is achieved for  $i=r$, we have
   $$\mathcal{O}^\lambda*\mathcal{O}^\mu=q^{d_{\rm min}}\mathcal{O}^{\lambda\uparrow r}*\mathcal{O}^{\mu\uparrow (n-r)}.$$
  By Proposition \ref{Belformula}, ${d_{\rm min}}$ is the smallest power of $q$ appearing in $[X^\lambda]\star[X^\mu]$ in $QH^*(X)$, and
   $[X^\lambda]\star[X^\mu]=q^{d_{\rm min}}[X^{\lambda\uparrow r}]\star[X^{\mu\uparrow (n-r)}]$.  Therefore there exist $\nu, \hat \nu\in \mathcal{P}_{k, n}$ such that
    $c_{\lambda\uparrow r, \mu\uparrow (n-r)}^{\hat \nu, 0}=c_{\lambda, \mu}^{\nu, d_{\rm min}}\neq 0$.
    Since $K(X)$ has a $\mathbb{Z}$-filtration structure $\{\bigoplus_{|\lambda|\geq k}\mathbb{Z}\mathcal{O}^{\lambda}\}_{k\in \mathbb{Z}}$  whose associated grading ring gives $H^*(X, \mathbb{Z})$, we have
      $N_{\lambda\uparrow r, \mu\uparrow (n-r)}^{\hat \nu, 0}=c_{\lambda\uparrow r, \mu\uparrow (n-r)}^{\hat \nu, 0}\neq 0$.
     It says that the smallest power of $q$ appearing in $\mathcal{O}^{\lambda\uparrow r}*\mathcal{O}^{\mu\uparrow (n-r)}$ is zero. Thus $d_{\rm min}$ must also be the smallest power of $q$ appearing in $\mathcal{O}^{\lambda}*\mathcal{O}^{\mu}$.
\end{proof}
\subsection{More reductions of quantum-to-classical type} Here we provide more reductions, especially   Theorem \ref{thmdegonerestated}, by using the operators $\mathcal{H}$ and $\mathcal{T}$ on $QK(X)$.
\begin{lemma}\label{lemred}
   Let $\lambda, \mu, \nu\in \mathcal{P}_{k, n}$ and $d\geq 1$.
   \begin{enumerate}
     \item If  $|\lambda|-|\lambda\uparrow 1|> |\nu|-|\nu\uparrow 1|$, then $N_{\lambda, \mu}^{\nu, d}=N_{\lambda\uparrow1, \mu}^{\nu\uparrow 1, d-1}$.
     \item If  $|\lambda|-|\lambda\downarrow 1|> |\nu|-|\nu\downarrow 1|$, then $N_{\lambda, \mu}^{\nu, d}=N_{\lambda\downarrow1, \mu}^{\nu\downarrow 1, d-1}$.
      \item If  $|\lambda|-|\lambda\uparrow i|= |\nu|-|\nu\uparrow i|$ for some $i$, then $N_{\lambda, \mu}^{\nu, d}=N_{\lambda\uparrow i, \mu}^{\nu\uparrow i, d}$.
     \item If  $|\lambda|-|\lambda\downarrow i|= |\nu|-|\nu\downarrow i|$ for some $i$, then $N_{\lambda, \mu}^{\nu, d}=N_{\lambda\downarrow i, \mu}^{\nu\downarrow i, d}$.
   \end{enumerate}
\end{lemma}
\begin{proof} Expanding  $\mathcal{T}^i(\mathcal{O}^\lambda*\mathcal{O}^\mu)$ in two ways, we have
   $$\mathcal{T}^i(\mathcal{O}^\lambda*\mathcal{O}^\mu)=\mathcal{T}^i(\sum_{\nu, d}N_{\lambda, \mu}^{\nu, d}\mathcal{O}^\nu q^d)
               =\sum_{\nu, d}N_{\lambda, \mu}^{\nu, d}q^d\mathcal{T}^i(\mathcal{O}^\nu)
               =\sum_{\nu, d}N_{\lambda, \mu}^{\nu, d}q^dq^{{i k+|\nu|-|\nu\uparrow i|\over n}}\mathcal{O}^{\nu\uparrow i};
   $$
   $$\mathcal{T}^i(\mathcal{O}^\lambda*\mathcal{O}^\mu)=\mathcal{T}^i(\mathcal{O}^\lambda)*\mathcal{O}^\mu=
   q^{{i k+|\lambda|-|\lambda\uparrow i|\over n}}\mathcal{O}^{\lambda\uparrow i}*\mathcal{O}^\mu
               =q^{{i k+|\lambda|-|\lambda\uparrow i|\over n}}\sum_{\nu, d}N_{\lambda\uparrow i, \mu}^{\nu, d}\mathcal{O}^\nu q^d.
    $$
  Hence, statement (3) follows immediately. By definition, $|\lambda|-|\lambda\uparrow 1|=-k$ if $\lambda_1<n-k$, or $n-k$ if $\lambda_1=n-k$; so does $|\nu|-|\nu\uparrow 1|$. Hence,  $|\lambda|-|\lambda\uparrow 1|> |\nu|-|\nu\uparrow 1|$ if and only if $|\lambda|-|\lambda\uparrow 1|=n-k$ and $|\nu|-|\nu\uparrow 1|=-k$. Comparing the above equalities for such $i=1$, we have
      $\sum_{\nu, d}N_{\lambda, \mu}^{\nu, d}q^d\mathcal{O}^{\nu\uparrow 1}= q\sum_{\nu, d}N_{\lambda\uparrow 1, \mu}^{\nu, d}\mathcal{O}^\nu q^d$.
      Therefore, $N_{\lambda, \mu}^{\nu, d}=N_{\lambda\uparrow1, \mu}^{\nu\uparrow 1, d-1}$, namely statement (1) holds.
     Similarly, we conclude statements (4) and (2), by using the associativity
       $\mathcal{H}^i(\mathcal{O}^\lambda*\mathcal{O}^\mu)=\mathcal{H}^i(\mathcal{O}^\lambda)*\mathcal{O}^\mu$.
\end{proof}
\begin{lemma}[Theorem 5.13 of \cite{BuMi}]\label{lemdual}
   For $\lambda, \mu, \nu\in \mathcal{P}_{k, n}$ and $d\geq 0$, $N_{\lambda, \mu}^{\nu, d}=N_{\lambda, \nu^\vee}^{\mu^\vee, d}$.
\end{lemma}

\begin{lemma}\label{lemcom}
    Let $\lambda, \nu\in \mathcal{P}_{k, n}$ and $1\leq m\leq k$. If $\nu_i\geq \lambda_i$ for all $1\leq i<m$ and $\nu_m<\lambda_m$, then we have
  \begin{align*}
      \lambda\uparrow(n-k-\lambda_m+m)&=({\scriptstyle \lambda_{m+1}+n-k-\lambda_{m},\cdots, \lambda_k+n-k-\lambda_m,\lambda_1-\lambda_m,\cdots, \lambda_{m-1}-\lambda_m,0}); \\
      \nu\uparrow(n-k-\lambda_m+m)&=({\scriptstyle \nu_{m}+n-k-\lambda_{m}+1,\cdots, \nu_k+n-k-\lambda_m+1,\nu_1-\lambda_m+1,\cdots, \nu_{m-1}-\lambda_m+1}).
   \end{align*}
\end{lemma}
\begin{proof}
   Denote $a_0=0$ and $\lambda^{(0)}=\lambda$. Set $a_i=n-k-\lambda_i+i$ and recursively define
     $\lambda^{(i)}=\lambda^{(i-1)}\uparrow (a_i-a_{i-1})$ for $1\leq i\leq m$. By the definition of Seidel shifts and induction on $i$,
   we conclude
    $$\lambda^{(i)}=(\lambda_{i+1}+a_i-i,\cdots, \lambda_{k-1}+a_i-i,\lambda_k+a_i-i, a_i-a_1-(i-1), a_i-a_2-(i-2),\cdots, a_i-a_{i-1}-1, 0)$$
    for all $1\leq i\leq m$. In particular,  the first half of the statement follows from the case $i=m$.

     Denote $b_0=0$ and $\nu^{(0)}=\nu$. Set $b_m=a_m$ and $b_i=n-k-\nu_i+i$ for $1\leq i\leq m-1$.
     Then, $b_m-b_{m-1}=\nu_{m-1}-\lambda_m+1\geq \lambda_{m-1}-\lambda_m+1>0$.
    For $1\leq i\leq m-1$, we recursively define
     $\nu^{(i)}=\nu^{(i-1)}\uparrow (b_i-b_{i-1})$, and by induction we conclude
     $$\nu^{(i)}=(\nu_{i+1}+b_i-i,\cdots, \lambda_{k-1}+b_i-i,\lambda_k+b_i-i, b_i-b_1-(i-1), b_i-b_2-(i-2),\cdots, b_i-b_{i-1}-1, 0).$$
  Noting $\nu_{m}+b_{m-1}-(m-1)+(b_m-b_{m-1})
  =\nu_m+n-k-\lambda_m+1\leq n-k$, we have
   $$ \nu\uparrow(n-k-\lambda_m+m)= \nu^{(m-1)}\uparrow (b_m-b_{m-1})=\nu^{(m-1)}+(b_m-b_{m-1}, \cdots, b_m-b_{m-1}).
      $$
   Therefore the second half of the statement follows.
\end{proof}
\begin{thm}\label{thmdegonerestated}
 Let $\lambda, \mu, \nu\in \mathcal{P}_{k, n}$ and $d\geq 1$.
 If $\nu_i<\lambda_i$ for some $i$, then we set $m=\min\{i\mid \nu_i<\lambda_i, 1\leq i\leq k\}$. We have
  \begin{align*}
     N_{\lambda, \mu}^{\nu, d}&=N_{\lambda \uparrow (n-k-\lambda_m+m), \mu}^{\nu\uparrow (n-k-\lambda_m+m), d-1}
 \\
     &=     N_{(\lambda_{m+1}+n-k-\lambda_{m},\cdots, \lambda_k+n-k-\lambda_m,\lambda_1-\lambda_m,\cdots, \lambda_{m-1}-\lambda_m,0),\mu}^{(\nu_{m}+n-k-\lambda_{m}+1,\cdots, \nu_k+n-k-\lambda_m+1,\nu_1-\lambda_m+1,\cdots, \nu_{m-1}-\lambda_m+1),d-1}.
   \end{align*}
\end{thm}
\begin{proof}
  Set $r=n-k-\lambda_m+m$ and  $t:={rk+|\lambda|-|\lambda\uparrow r|\over n}-{rk+|\nu|-|\nu\uparrow r|\over n}$.
   By comparing the equalities for the two ways of expansions of   $\mathcal{T}^r(\mathcal{O}^\lambda*\mathcal{O}^\mu)$(as in the proof of Lemma \ref{lemred}),
  we have $$N_{\lambda, \mu}^{\nu, d}= N_{\lambda \uparrow r, \mu}^{\nu\uparrow r, d-t}.$$
  By Lemma \ref{lemcom}, we have $|\lambda|-|\lambda\uparrow r|=-(n-k)(k-m)+k\lambda_m$ and
  $|\nu|-|\nu\uparrow r|=-(n-k+1)(k-m+1)+k\lambda_m-m+1$. Hence, $t=1$, namely the first equality in the statement holds. By Lemma \ref{lemcom}, the second equality   holds as well.
\end{proof}

We provide two   consequences of Theorem \ref{thmdegonerestated} below. 

\begin{prop}\label{prophigherdegp}
 Let $\lambda, \mu, \nu\in \mathcal{P}_{k, n}$ and $d\geq s\geq 2$.
  Suppose $\nu_1+s-2<\lambda_{s-1}$ and $\nu_{j-s+1}+s-1<\lambda_j$ for some integer $j\in [s, k]$. Let $t=\min
  \{j\mid  \nu_{j-s+1}+s-1<\lambda_j, s\leq j\leq k\}$. Then we have
   \begin{align*}
     N_{\lambda, \mu}^{\nu, d}&=N_{\lambda \uparrow (n-k- \lambda_t+t), \mu}^{\nu\uparrow (n-k- \lambda_t +t), d-s}.
    \end{align*}

\end{prop}

    \begin{proof}
     Since $\nu_1+i\leq \nu_1+s-2 <\lambda_{s-1}\leq \lambda_{i+1}$ for all $0\leq i\leq s-2$, by applying Theorem \ref{thmdegonerestated} repeatedly, we have
      \begin{eqnarray*}
         N_{\lambda, \mu}^{\nu, d}\hspace{-0.6cm}&=N_{\lambda\uparrow (n-k-\lambda_1+1), \mu}^{\nu\uparrow (n-k-\lambda_1+1), d-1}&=N_{(n-k+\lambda_2-\lambda_1, \cdots, n-k+\lambda_k-\lambda_1, 0),\mu}^{(n-k+\nu_1-\lambda_1+1, \cdots, n-k+\nu_k-\lambda_1+1), d-1}\\
          &=N_{\lambda\uparrow (n-k-\lambda_2+2), \mu}^{\nu\uparrow (n-k-\lambda_2+2), d-2}&=N_{(n-k+\lambda_3-\lambda_2, \cdots, n-k+\lambda_k-\lambda_2, \lambda_1-\lambda_2,0),\mu}^{(n-k+\nu_1-\lambda_2+2, \cdots, n-k+\nu_k-\lambda_2+2), d-2}\\
           &=N_{\lambda\uparrow (n-k-\lambda_{s-1}+s-1), \mu}^{\nu\uparrow (n-k-\lambda_{s-1}+s-1), d-s+1}&=:N_{\bar \lambda, \bar\mu}^{\bar \nu, d-s+1}\\
      \end{eqnarray*}
      with $N_{\bar \lambda, \bar\mu}^{\bar \nu, d-s+1}=N_{(n-k+\lambda_s-\lambda_{s-1}, \cdots, n-k+\lambda_k-\lambda_{s-1}, \lambda_1-\lambda_{s-1},\cdots, \lambda_{s-2}-\lambda_{s-1}, 0),\mu}^{(n-k+\nu_1-\lambda_{s-1}+s-1, \cdots, n-k+\nu_k-\lambda_{s-1}+{s-1}), d-s+1}.$ Then we have
        $N_{\bar \lambda, \bar\mu}^{\bar \nu, d-s+1}=N_{\bar \lambda\uparrow (n-k-\bar{\lambda}_{m}+m), \bar\mu}^{\bar \nu\uparrow (n-k-\bar \lambda_m+m), d-s}= N_{\lambda \uparrow (n-k- \lambda_t+t), \mu}^{\nu\uparrow (n-k- \lambda_t +t), d-s}$,
          again by Theorem \ref{thmdegonerestated} with $m=t-s+1$.
     \end{proof}

\begin{example}
{\upshape
  In $QK(Gr(6,17))$,
    take $\lambda=\mu=(10, 8, 6,4,2,0)$ and $d=3$.

 For $\nu =(3,3,2,1,0,0)$,   $\nu_1+3-2<\lambda_{3-1}$ and $\nu_{3-3+1}+3-1<\lambda_{3}$ hold. By Proposition \ref{prophigherdegp} with respect to $s=t=3$, we have $n-k-\lambda_t+t=8$ and
   $$ N_{\lambda, \mu}^{\nu, 3} =N_{\lambda \uparrow 8, \mu}^{\nu\uparrow 8, 0}=N_{(9, 7, 5, 3, 2), (10, 8, 6, 4, 2, 0)}^{(11,11,10,9,8,8), 0}.$$

   For  $\eta=(6,2,2,1,0,0)$,    Proposition \ref{prophigherdegp}  is not applicable. Nevertheless, we can apply   Theorem \ref{thmdegonerestated} repeatedly and obtain
   \begin{align*}
      N_{(10, 8, 6,4,2,0), (10, 8, 6,4,2,0)}^{(6,2,2,1,0,0), 3}
      &=
             N_{(9, 7, 5,3,1,0), (10, 8, 6,4,2,0)}^{(8, 4, 4,3,2,2), 2}\\
             &=N_{ (9, 7, 5,3,1,0), (9, 7, 5,3,1,0)}^{(10,6,6,5, 4, 4), 1}\\
             &=N_{ (9, 7, 5,4,2,0), (9, 7, 5,3,1,0)}^{(11,11,10,9, 9, 4), 0}.\\
   \end{align*}
}
  \end{example}

\begin{prop}\label{propred33}
 Let $\lambda, \mu, \nu\in \mathcal{P}_{k, n}$ and $d \geq 1$. If $\nu_1\geq \lambda_1$ and $\nu_1<\lambda_{k+1-j}+\mu_{j}$ for some $j$, then we set $m=\min
  \{j\mid  \nu_1<\lambda_{k+1-j}+\mu_{j}, 1\leq j\leq k\}$. We have
   $$N_{\lambda, \mu}^{\nu, d}=
   N_{\nu^\vee\downarrow (n-k-\nu_1), \mu \downarrow (k+\mu_m-m)}^{\lambda^\vee\downarrow(n-\nu_1+\mu_m-m), d-1} .$$
\end{prop}
\begin{proof}
  Since $\nu_1\geq \lambda_1$, we have $n-k-\lambda_1\geq n-k-\nu_1$.   By Lemma \ref{lemdual} and Lemma \ref{lemred} (4), we have

   $$N_{\lambda, \mu}^{\nu, d}=N_{\nu^\vee, \mu}^{\lambda^\vee, d}=N_{\nu^\vee\downarrow (n-k-\nu_1), \mu}^{\lambda^\vee\downarrow(n-k-\nu_1), d}
     =N_{(\nu_1-\nu_k, \cdots, \nu_1-\nu_2, 0), (\mu_1, \cdots, \mu_k)}^{(\nu_1-\lambda_k, \cdots, \nu_1-\lambda_1), d}=:N_{\bar \lambda, \mu}^{\bar \nu,d}.$$
  Therefore by Theorem \ref{thmdegonerestated} together with the definition of Seidel shifts, we have
  $$N_{\lambda, \mu}^{\nu, d}=N_{\bar \lambda, \mu\uparrow (n-k-\mu_m+m)}^{\bar \nu \uparrow (n-k-\mu_m+m),d-1}=N_{\bar \lambda, \mu\downarrow (k+\mu_m-m)}^{\bar \nu \downarrow (k+\mu_m-m),d-1}=
   N_{\nu^\vee\downarrow (n-k-\nu_1), \mu \downarrow (k+\mu_m-m)}^{\lambda^\vee\downarrow(n-\nu_1+\mu_m-m), d-1} .$$
\end{proof}

\section{A quantum Littlewood-Richardson rule for   $QK(Gr(3, n))$}
In this section, we restrict to $Gr(3, n)$, and provide a quantum Littlewood-Richardson rule for $QK(Gr(3, n))$ in \textbf{Theorem \ref{thmQLRforGr3n}}. We obtain the rule by reducing most of $N_{\lambda, \mu}^{\nu, d}$  to  corresponding $N_{\tilde \lambda, \tilde \mu}^{\tilde \nu, 0}$ and computing the rest directly. As a direct consequence, we show the alternating positivity in Corollary \ref{positivityGr3n}.

\begin{prop}\label{propred2zero}
    Let $\lambda, \mu, \nu\in \mathcal{P}_{3, n}$ and $d\geq 1$. In $QK(Gr(3, n))$,  we have
$$N_{\lambda, \mu}^{\nu, d}=N_{(\lambda_1-\lambda_3, \lambda_2-\lambda_3, 0), (\mu_1-\mu_3, \mu_2-\mu_3,0)}^{\nu\downarrow (\lambda_3+\mu_3), d+{|\nu|-|\nu\downarrow (\lambda_3+\mu_3)|-3(\lambda_3+\mu_3)\over n}}=:N_{\hat \lambda, \hat \mu}^{\hat \nu, \hat d}.$$
Moreover, we have  $(-1)^{|\lambda|+|\mu|+|\nu|+dn}=(-1)^{|\hat\lambda|+|\hat\mu|+|\hat\nu|+\hat dn}$.
\end{prop}
\begin{proof}
  The arguments for the first equality are the same as that for Lemma \ref{lemred}. The second equality follows immediately from the definition of the notation.
\end{proof}

Thanks to the above proposition, we will always consider the  partitions $\lambda=(\lambda_1, \lambda_2, 0)$ and $\mu=(\mu_1, \mu_2, 0)$ in $\mathcal{P}_{3, n}$ in the rest of this section.

The next lemma is a special case of the Giambelli formula in \cite[Theomrem 5.6]{BuMi}. Here we provide  the detail by  quantum Pieri formula for completeness.
\begin{lemma}\label{lemGiamb}
 Let $\mu=(\mu_1, \mu_2, 0)\in \mathcal{P}_{3, n}$. In $QK(Gr(3, n))$, we have
   \begin{align*}
            \mathcal{O}^{\mu}&= \mathcal{O}^{\mu_1}*\mathcal{O}^{\mu_2-1}+\sum_{j=\mu_1}^{n-3}\mathcal{O}^{j}*(\mathcal{O}^{\mu_2}-\mathcal{O}^{\mu_2-1}).
          \end{align*}
  \end{lemma}
 \begin{proof}
    For any $n-3\geq a\geq b\geq 0$, by the quantum Pieri rule we have
      $$\mathcal{O}^a*\mathcal{O}^b=\mathcal{O}^{(a, b,0)}+\sum_{j=a+1}^{\min\{a+b, n-3\}}(\mathcal{O}^{(j, a+b-j,0)}-\mathcal{O}^{(j, a+b-j+1,0)}).$$
Hence, the right hand side RHS of the expected equality satisfies
\begin{align*}
    \mbox{RHS}&
 =\mathcal{O}^{n-3}* \mathcal{O}^{\mu_2}+\sum_{j={\mu_1}}^{n-4}(\mathcal{O}^j*\mathcal{O}^{\mu_2}-\mathcal{O}^{j+1}*\mathcal{O}^{\mu_2-1})\\
 &=\mathcal{O}^{(n-3,\mu_2,0)}+\sum_{j={\mu_1}}^{n-4}(\mathcal{O}^{(j,\mu_2,0)}-\mathcal{O}^{(j+1, \mu_2,0)})\\
 &=\mathcal{O}^{(\mu_1, \mu_2, 0)}.
\end{align*}

 \end{proof}

\begin{thm}\label{thmQLRforGr3n}
Let $\lambda,\mu, \nu\in \mathcal{P}_{3,n}$ with $\lambda_3=\mu_3=0$. In $QK(Gr(3,n))$,
  we have $N_{\lambda, \mu}^{\nu, d}=0$  unless $d\leq 1$. Moreover, descriptions of $N_{\lambda, \mu}^{\nu, 1}$ are given as follows.

 \begin{enumerate}
   \item If $\nu_1<\max\{\lambda_1,\mu_1\}$,  assuming $\nu_1<\lambda_1$,  we have  $$N_{\lambda, \mu}^{\nu, 1}=N_{(\lambda_2+n-3-\lambda_1, n-3-\lambda_1, 0), \mu}^{(\nu_1+n-2-\lambda_1, \nu_2+n-2-\lambda_1,\nu_3+n-2-\lambda_1), 0}.$$

   \item If $\nu_1\geq \max\{\lambda_1,\mu_1\}$ and $\nu_2< \max\{\lambda_2,\mu_2\}$,   assuming $\nu_2<\lambda_2$,  we have
           $$N_{\lambda, \mu}^{\nu, 1}=N_{(n-3-\lambda_2, \lambda_1-\lambda_2, 0), \mu}^{(\nu_2+n-2-\lambda_2, \nu_3+n-2-\lambda_2,\nu_1-\lambda_2+1), 0}.$$
    \item If $\nu_1\geq \max\{\lambda_1, \mu_1\}$ and $\nu_2\geq\max\{\lambda_2, \mu_2\}$, setting
      $m:=|\nu|+n-|\lambda|-|\mu|$ and $A:=\lambda_1+\mu_1-\nu_1-\nu_2$, we have
       $$N_{\lambda, \mu}^{\nu, 1}=\begin{cases} \min\{A-1, n-3-\nu_1\},&\mbox{if } m=0,\\
   -\min\{A, n-3-\nu_1\}-2\min\{A-1, n-3-\nu_1\},&\mbox{if } m=1,\\
   2\min\{A, n-3-\nu_1\}+\min\{A-1, n-3-\nu_1\},&\mbox{if } m=2,\\
   -\min\{A, n-3-\nu_1\} ,&\mbox{if } m=3,\\
\end{cases}$$
   provided all the   constraints in \eqref{constraint} hold, or  $N_{\lambda, \mu}^{\lambda, 1}=0$ otherwise.
  \begin{eqnarray}\label{constraint}
     \begin{cases}
      A>0,\\ 0\leq m\leq 3,\\
      \min\{\lambda_1+\lambda_2, \mu_1+\mu_2\}\geq n-3+\nu_3,\\
      \min\{\lambda_1, \mu_1\}>\nu_{2}, \quad \min\{\lambda_2, \mu_2\}>\nu_{3}.
    \end{cases}
    \end{eqnarray}

 \end{enumerate}
\end{thm}

 \begin{proof}
  By Lemma \ref{lemGiamb},
   \begin{align*}
            \mathcal{O}^{\lambda}*\mathcal{O}^{\mu}&=\mathcal{O}^{(\lambda_1,\lambda_2,0)}*\big(\mathcal{O}^{\mu_1}*\mathcal{O}^{\mu_2-1}+\sum\nolimits_{j=\mu_1}^{n-3}\mathcal{O}^{j}*(\mathcal{O}^{\mu_2}-\mathcal{O}^{\mu_2-1})\big)\\
            &=\mathcal{O}^{(\lambda_1,\lambda_2,0)}*\mathcal{O}^{\mu_1}*\mathcal{O}^{\mu_2-1}+\big(\sum\nolimits_{j=\mu_1}^{n-3}\mathcal{O}^{(\lambda_1,\lambda_2,0)}*\mathcal{O}^{j}\big)*(\mathcal{O}^{\mu_2}-\mathcal{O}^{\mu_2-1}).
  \end{align*}
Since $\lambda_3=0$, we have  $\mathcal{O}^{(\lambda_1,\lambda_2,0)}* \mathcal{O}^{j}=\mathcal{O}^{(\lambda_1,\lambda_2,0)}\cdot \mathcal{O}^{j}=\sum_{\eta} c_\eta \mathcal{O}^\eta$. Thus $N_{\lambda, \mu}^{\nu, d}=0$ whenever $d>1$, by Proposition \ref{propQuanPieri}. Namely the first claim in the statement holds.

 Statement (1) and (2) follow directly from Theorem \ref{thmdegonerestated}.

 Now we start with the hypotheses $\nu_1\geq \lambda_1$ and $\nu_2\geq \lambda_2$ only.

For any $a$, the term $q \mathcal{O}^{\nu}$ occurs in $\mathcal{O}^\eta*\mathcal{O}^a$ only if
 $\nu$ can be obtained from $\eta$ by removing at least one box at each row.
 In particular if $\lambda_i\leq \nu_{i+1}$ for some $1\leq i\leq 2$, then none of $\eta$ could satisfy the requirement. Consequently we have
   \begin{equation}\label{vanishing111}
      N_{\lambda, \mu}^{\nu, 1}=0, \mbox{if } \nu_i\geq \lambda_i \mbox{ for all } i\in\{1, 2\} \mbox{ and } \lambda_j\leq \nu_{j+1} \mbox{ for some }j\in \{1,2\}.
   \end{equation}
 Since $\nu_1\geq \lambda_1$ and $\nu_2\geq \lambda_2$, it follows that
$\eta_i > \lambda_i$ for $1\leq i\leq 3$, so that   $r(\eta/\lambda)=3$. Thus by Proposition \ref{propQuanPieri},
  $N_{\lambda, \mu}^{\nu, 1}$ coincides with the coefficient of $q\mathcal{O}^\nu$ in PRHS with

    \begin{align*}
     \mbox{PRHS}  &=\,\,\sum_{i=0}^2\sum_{ |\eta/\lambda|=\mu_1+i }(-1)^i{2\choose i}\mathcal{O}^{\eta}*\mathcal{O}^{\mu_2-1}\\
     &\qquad+ \sum\nolimits_{j=\mu_1}^{n-3}\sum_{i=0}^2\sum_{|\eta/\lambda|=j+i }(-1)^i{2\choose i}\mathcal{O}^{\eta} *(\mathcal{O}^{\mu_2}-\mathcal{O}^{\mu_2-1})\\
      &=\,\,\sum_{i=0}^2\sum_{ |\eta/\lambda|=\mu_1+i }(-1)^i{2\choose i}\mathcal{O}^{\eta}*\mathcal{O}^{\mu_2-1}\\
     &\qquad+\big(\sum_{|\eta/\lambda|=\mu_1}\mathcal{O}^{\eta}-\sum_{|\eta/\lambda|=\mu_1+1}\mathcal{O}^{\eta}\big)*(\mathcal{O}^{\mu_2}-\mathcal{O}^{\mu_2-1})\\
      &= \sum_{ |\eta/\lambda|=\mu_1+2 } \mathcal{O}^{\eta}*\mathcal{O}^{\mu_2-1}-\sum_{ |\eta/\lambda|=\mu_1+1 }  \mathcal{O}^{\eta}*\mathcal{O}^{\mu_2-1}\\
     &\qquad -\sum_{|\eta/\lambda|=\mu_1+1 }\mathcal{O}^{\eta} * \mathcal{O}^{\mu_2}+ \sum_{|\eta/\lambda|=\mu_1 }\mathcal{O}^{\eta}*\mathcal{O}^{\mu_2}\\ 
    \end{align*}
    Since $\eta/\lambda$ is a horizontal strip, $|\eta/\lambda|\leq n-3$  is an  implicit constraint  in the above sums. In particular if
    $|\eta/\lambda|>n-3$, then the corresponding sum is read off as 0.
 Since $r(\eta/\lambda)=3$, the outer rim of $\eta$ has 3 rows.  Since $\nu_i\geq \lambda_i$ for all $i$,    the first two rows of the outer rim of $\eta$ both contain at least a box in $\nu$. Therefore for any $a$, the coefficient of $q\mathcal{O}^\nu$ in $\mathcal{O}^\eta*\mathcal{O}^a$  is equal to $(-1)^e{2\choose e}$, provided
       $|\nu|-|\eta|+n-a=e\in \{0, 1, 2\}$ and the implicit constraint that $\nu$ can be obtained from $\eta$ by removing a subset of the boxes in the outer rim of $\eta$ with at least one box removed from each row. Hence,
\begin{eqnarray*}\label{degone}
     N_{\lambda, \mu}^{\nu, 1}  &=\qquad \sum_{ |\eta/\lambda|=\mu_1+2 \atop |\nu|-|\eta|+n-\mu_2+1=0}1+
   \sum_{ |\eta/\lambda|=\mu_1+2 \atop |\nu|-|\eta|+n-\mu_2+1=1}(-2) +
   \sum_{ |\eta/\lambda|=\mu_1+2 \atop |\nu|-|\eta|+n-\mu_2+1=2}1\\
   &  \qquad -  \sum_{ |\eta/\lambda|=\mu_1+1 \atop |\nu|-|\eta|+n-\mu_2+1=0}1-
   \sum_{ |\eta/\lambda|=\mu_1+1 \atop |\nu|-|\eta|+n-\mu_2+1=1}(-2) -
   \sum_{ |\eta/\lambda|=\mu_1+1 \atop |\nu|-|\eta|+n-\mu_2+1=2}1\\
    &\quad - \sum_{ |\eta/\lambda|=\mu_1+1 \atop |\nu|-|\eta|+n-\mu_2=0}1-
   \sum_{ |\eta/\lambda|=\mu_1+1 \atop |\nu|-|\eta|+n-\mu_2=1}(-2) -
   \sum_{ |\eta/\lambda|=\mu_1+1 \atop |\nu|-|\eta|+n-\mu_2=2}1\\
   &\quad  +\sum_{ |\eta/\lambda|=\mu_1 \atop |\nu|-|\eta|+n-\mu_2=0}1+
   \sum_{ |\eta/\lambda|=\mu_1 \atop |\nu|-|\eta|+n-\mu_2=1}(-2) +
   \sum_{ |\eta/\lambda|=\mu_1 \atop |\nu|-|\eta|+n-\mu_2=2}1.
 \end{eqnarray*}
  Let $m=n+|\nu|-|\lambda|-|\mu|$. The first sum occurs only if $|\eta/\lambda|+(|\nu|-|\eta|+n-\mu_2+1)=(\mu_1+2)+0$, namely only if $m=1$. For the same reason, only part of the sum could occur at the same time. Precisely, we have
 \begin{eqnarray}\label{degonered}
 \begin{split}{}\hspace{1.5cm}&N_{\lambda, \mu}^{\nu, 1}\\
    =&\begin{cases}\displaystyle
             -\sum_{ |\eta/\lambda|=\mu_1+1 \atop |\nu|-|\eta|+n-\mu_2+1=0}1+\sum_{ |\eta/\lambda|=\mu_1 \atop |\nu|-|\eta|+n-\mu_2=0}1, &\mbox{if } m=0,\\
         \displaystyle
 \sum_{ |\eta/\lambda|=\mu_1+2 \atop |\nu|-|\eta|+n-\mu_2+1=0}1  + \sum_{ |\eta/\lambda|=\mu_1+1 \atop |\nu|-|\eta|+n-\mu_2+1=1}1+ \sum_{ |\eta/\lambda|=\mu_1 \atop |\nu|-|\eta|+n-\mu_2=1}(-2), &\mbox{if } m=1,\\
 \displaystyle\sum_{ |\eta/\lambda|=\mu_1+2 \atop |\nu|-|\eta|+n-\mu_2+1=1}(-2)+ \sum_{ |\eta/\lambda|=\mu_1+1 \atop |\nu|-|\eta|+n-\mu_2=1}1+ \sum_{ |\eta/\lambda|=\mu_1 \atop |\nu|-|\eta|+n-\mu_2=2}1, &\mbox{if } m=2,\\
\displaystyle  \sum_{ |\eta/\lambda|=\mu_1+2 \atop |\nu|-|\eta|+n-\mu_2+1=2}1-\sum_{ |\eta/\lambda|=\mu_1+1 \atop |\nu|-|\eta|+n-\mu_2=2}1, &\mbox{if } m=3,\\
0,&\mbox{otherwise}.          \end{cases}
 \end{split}
  \end{eqnarray}

 Notice  $|\nu|-|\lambda|\geq \nu_3$. If $\mu_1+\mu_2<n-3+\nu_3$, then
   $m\geq \nu_3+n-|\mu|=\nu_3+n-(\mu_1+\mu_2)>3$.
  Therefore  by formula \eqref{degonered}, we have
  \begin{equation}\label{vanishing222}
      N_{\lambda, \mu}^{\nu, 1}=0, \mbox{if } \nu_i\geq \lambda_i \mbox{ for all } i\in\{1, 2\} \mbox{ and } \mu_1+\mu_2<n-3+\nu_3.
   \end{equation}
   The above arguments only use the hypotheses $\nu_1\geq \lambda_1$ and $\nu_2\geq \lambda_2$. Since $\nu_i\geq \max\{\lambda_i, \mu_i\}$, by interchanging $\lambda$ and $\mu$ in \eqref{vanishing111} and \eqref{vanishing222}, we obtain
    \begin{eqnarray}\label{vanishing333}
     {} \qquad N_{\lambda, \mu}^{\nu, 1}&=0,& \mbox{if } \nu_i\geq \mu_i \mbox{ for all } i\in\{1, 2\}  \mbox{ and } \lambda_j\leq \nu_{j+1} \mbox{ for some }j\in \{1,2\};\\
      {}\qquad N_{\lambda, \mu}^{\nu, 1}&=0,& \mbox{if } \nu_i\geq \mu_i \mbox{ for all } i\in\{1, 2\} \mbox{ and  } \lambda_1+\lambda_2<n-3+\nu_3.
   \end{eqnarray}

 Now we consider the case $\mu_1+\mu_2\geq n-3+\nu_3$ and $\lambda_j>\nu_{j+1} \mbox{ for all }j\in \{1,2\}$. Assume    $\lambda_2\geq \mu_2$ (otherwise we interchange $\lambda$ and $\mu$ in the   arguments below).  Set
   $$\Gamma_{i}^j:=\{\eta\mid \eta \mbox{ occurs in the unique sum  when } m=j \mbox{ and } |\eta/\lambda|=\mu_1+i\}.$$
   We claim $(\eta_1, \eta_2, \eta_3)\mapsto (\eta_1, \eta_2, \eta_3-1)$ well defines a map $\psi_i^j: \Gamma_{i}^j\to\Gamma_{i-1}^j$ for all $(i, j)$.
   Indeed, we just need to show  $\nu$ can   be obtained from $(\eta_1, \eta_2, \eta_3-1)$   by removing a subset of the boxes in the outer rim of $(\eta_1, \eta_2, \eta_3-1)$ with at least one box removed from each row, for all $\eta\in \Gamma_i^j$.
     That is, we need to show   $\eta_3-1>\nu_3$.
   Notice   $\mu_1+i=|\eta|-|\lambda|=\eta_3+(\eta_1-\lambda_1)+(\eta_2-\lambda_2)\leq \eta_3+(n-3-\lambda_1)+(\lambda_1-\lambda_2)$. Hence
    $\eta_3 \geq   \mu_1+i+\lambda_2-(n-3)
         \geq \mu_1+\mu_2+i-(n-3) $.
   \begin{enumerate}
     \item[a)] If $m=3$, then we have $i=2$ and hence  $\eta_3\geq n-3+\nu_3+2-(n-3)=\nu_3+2$.
     \item[b)] If $m<3$, then $|\mu|=|\nu|+n-|\lambda|-m>\nu_3+n-3$, and consequently $\eta_3>\nu_3+ n-3+i-(n-3)=\nu_3+i\geq \nu_3+1$.
   \end{enumerate}
 Thus we   always have $\eta_3-1>\nu_3$. Since $\psi_i^j$ is  obviously injective,  we have
  \begin{equation}\label{NGr3n}
     N_{\lambda, \mu}^{\nu, 1}=\begin{cases} \big|\Gamma_{0}^m\setminus \psi_{1}^m(\Gamma_1^m)|, &\mbox{if }m=0,\\
     -\big|\Gamma_{1}^m\setminus \psi_{2}^m(\Gamma_2^m)|-2\big|\Gamma_{0}^m\setminus \psi_{1}^m(\Gamma_1^m)|, &\mbox{if }m=1,\\
      2\big|\Gamma_{1}^m\setminus \psi_{2}^m(\Gamma_2^m)|+ \big|\Gamma_{0}^m\setminus \psi_{1}^m(\Gamma_1^m)|, &\mbox{if }m=2,\\
      -\big|\Gamma_{1}^m\setminus \psi_{2}^m(\Gamma_2^m)|, &\mbox{if }m=3.
  \end{cases}
  \end{equation}

  \begin{align*}
       &\Gamma_{i-1}^m\setminus \psi_{i}^m(\Gamma_i^m)\\
       =&\{\eta\in \mathcal{P}_{3, n}\mid \eta_3=\lambda_2, \nu_1+1\leq \eta_1, \nu_2+1\leq \eta_2\leq \lambda_1, |\eta|-|\lambda|=\mu_1+i-1\}\\
          \overset{\cong}{=} &\{(a, b)\in \mathbb{Z}^2\left| {  \nu_1-\lambda_1+1\leq a\leq n-3-\lambda_1, \atop \nu_2+1-\lambda_2\leq b\leq \lambda_1-\lambda_2, \quad a+b=\mu_1+i-1-\lambda_2}\right.\}
          .\\
  \end{align*}
Since $\mu_1\leq \nu_1$ and $i\in \{1, 2\}$, we have
$$  \mu_1+i-1-\lambda_2 -(\nu_1-\lambda_1+1)=\lambda_1-\lambda_2+(\mu_1-\nu_1)+(i-2)\leq \lambda_1-\lambda_2;$$
  $$\mu_1+i-1-\lambda_2 -(\nu_1-\lambda_1+1)-(\nu_2-\lambda_2+1)=\lambda_1+\mu_1-\nu_1-\nu_2+i-3=A+i-3.$$
Thus $\Gamma_{i-1}^m\setminus \psi_{i}^m(\Gamma_i^m)\neq \emptyset$ if and only if $A+i-3\geq 0$. Hence, by \eqref{NGr3n} we have
$$N_{\lambda, \mu}^{\nu, 1}=0,\quad\mbox{ if }A\leq0.$$

Now we assume  $A>0$. (We allow the  case $A=1$ and $i=1$,
 since $\Gamma_{0}^m\setminus \psi_{1}^m(\Gamma_1^m)$ would be an empty set, consistent with the counting number $A-1=0$).
For $M=\mu_1-\nu_2+i-2$, we have $\mu_1+i-1-\lambda_2=M+(\nu_2+1-\lambda_2)$, and $M$ is the  upper bound of $a$ for which  the required lower bound of $b$ is satisfied. Hence,
$$\Gamma_{i-1}^m\setminus \psi_{i}^m(\Gamma_i^m)\cong\{a\in \mathbb{Z}\mid \nu_1-\lambda_1+1\leq a\leq \min\{M, n-3-\lambda_1\}\},$$
independent of the value $m$.
Hence, $$\big|\Gamma_{i-1}^m\setminus \psi_{i}^m(\Gamma_i^m)|
=\begin{cases}
    n-3-\nu_1\leq A+i-2,&\mbox{if } M\geq n-3-\lambda_1\\
      A+i-2\leq n-3-\nu_1,&\mbox{if } M< n-3-\lambda_1.
\end{cases}$$
   Consequently statement (3) follows from \eqref{NGr3n}.
  \end{proof}
As a direct consequence, we verify   \cite[Conjecture 5.10]{BuMi} by Buch and Mihalcea  in the special case $Gr(3, n)$.
  \begin{cor}\label{positivityGr3n}
     Let $\lambda,\mu, \nu\in \mathcal{P}_{3,n}$ and $d\in \mathbb{Z}_{\geq 0}$. In $QK(Gr(3, n))$, we have
        $$(-1)^{|\lambda|+|\mu|+|\nu|+dn}N_{\lambda, \mu}^{\nu, d}\geq 0.$$
  \end{cor}
  \begin{proof}
    If $d=0$, the alternating positivity is   known to hold  \cite{Buch,Brio}. By Proposition \ref{propred2zero}, we can assume $\lambda_3=\mu_3=0$ and $d=1$.

     If $\nu_i<\max\{\lambda_i, \mu_i\}$ for some $1\leq i\leq 2$, then it follows
        from Theorem \ref{thmQLRforGr3n} (1) and (2) that
        $N_{\lambda, \mu}^{\nu, 1}=N_{\bar \lambda, \bar \mu}^{\bar \nu, 0}$ with
         $|\lambda|+|\mu|+|\nu|+n\equiv |\bar\lambda|+|\bar\mu|+|\bar\nu|\mod 2$. Hence the statement follows.

     If $\nu_i\geq\max\{\lambda_i, \mu_i\}$  for all $1\leq i\leq 2$, then the statement follows directly from Theorem \ref{thmQLRforGr3n} (3).
  \end{proof}

\begin{example}\label{exampleccc}
Let $c, u$ be   integers with   $0\leq c< u\leq 2c< n-3$. Let $\lambda=\mu=(2c, c, 0)$ and $\mu=(u, c, 0)$. Then $A=
\lambda_1+\mu_1-\nu_1-\nu_2=u-c>0$. In  $QK(Gr(3,n))$,  by Theorem \ref{thmQLRforGr3n} (3) we have
\begin{align*}
   N_{(2c,c,0),(u,c,0)}^{(2c,c,0),1}=
\begin{cases}
n-3-2c,&\mbox{if } u=n-c,\\
-3(n-3-2c),&\mbox{if } u=n-1-c,\\
3(n-3-2c), & \mbox{if }u=n-2-c,\\
-(n-3-2c),&\mbox{if }u=n-3-c,\\
0,&\mbox{otherwise}.
\end{cases}
\end{align*}
In particular if $u=2c$, then $n-3=3c$ and $ N_{(2c,c,0),(2c,c,0)}^{(2c,c,0),1}=-c$.
\end{example}
We remark that part of the structure constants $N_{\lambda, \mu}^{\nu, 1}$  in case (3) in Theorem \ref{thmQLRforGr3n} can also be reduced to $N_{\tilde \lambda, \tilde \mu}^{\tilde \nu, 0}$. For instance if $\nu_1<\lambda_2+\mu_2$, then  the  reduction can be done by Proposition \ref{propred33}. Here we provide a reduction for the part when $\nu_3=0$.

\begin{prop}
Let $\lambda,\mu, \nu\in \mathcal{P}_{3,n}$ with $\lambda_3=\mu_3=0$.
Assume  $\nu_1\geq \max\{\lambda_1, \mu_1\}$, $\nu_2\geq\max\{\lambda_2, \mu_2\}$ and $\nu_3=0$.
 If either of   {\upshape i), ii)} holds,
 {\upshape $$  \mbox{i)  } n-3-\mu_{j}< \lambda_{3-j} \mbox{ for some }j\in \{1, 2\}; \mbox{ ii) } \lambda=\mu=\nu,  \lambda_1<2\lambda_2, \lambda_1+\lambda_{2}\leq n-3; $$
}
 then $N_{\lambda,\mu}^{\nu,1}= N_{\tilde{\lambda} ,\tilde{\mu} }^{\tilde{\nu} ,0}$ for some $\tilde\lambda,\tilde \mu, \tilde\nu\in \mathcal{P}_{3,n}$ with explicit descriptions.
 \begin{enumerate}
   \item[iii)] If $\lambda=\mu=\nu,  \lambda_1=2\lambda_2, \lambda_1+\lambda_{2}\leq n-3$,,  then $N_{\lambda,\mu}^{\nu,1}=-\lambda_2$.
  \end{enumerate}
If none of {\upshape i), ii), iii)} holds, then  $N_{\lambda,\mu}^{\nu,1}=0$.
 \end{prop}

\begin{proof}
   By Lemma \ref{lemdual} and Lemma \ref{lemred} (3), we have
    \begin{align*}
      N_{\lambda,\mu}^{\nu,1}=N_{(\lambda_1,\lambda_2,0),(n-3,n-3-\nu_2,n-3-\nu_1)}^{(n-3, n-3-\mu_2,n-3-\mu_1),1}=N_{(\lambda_1,\lambda_2,0),(n-3-\nu_2,n-3-\nu_1,0)}^{(n-3-\mu_2,n-3-\mu_1,0),1}.
    \end{align*}

    \begin{enumerate}
      \item Assume  $n-3-\mu_2<\lambda_1$, then by Theorem \ref{thmdegonerestated} we have
       $$N_{\lambda,\mu}^{\nu,1}=
      N_{(n-3-\lambda_1+\lambda_2,n-3-\lambda_1,0),(n-3-\nu_2,n-3-\nu_1,0)}^{(2n-5-\lambda_1-\mu_2,2n-5-\lambda_1-\mu_2,n-2-\lambda_1),0}
       .$$
      \item Assume   $n-3-\mu_2\geq \lambda_1$ and $n-3-\mu_1<\lambda_2$, then   by Theorem \ref{thmdegonerestated} we have
         $$N_{\lambda,\mu}^{\nu,1}=
      N_{(n-3-\lambda_2+\lambda_3,\lambda_1-\lambda_2,0),(n-3-\nu_2,n-3-\nu_1,0)}^{(2n-5-\lambda_2-\mu_1, n-2-\lambda_2,n-2-\mu_2-\lambda_2),0}
       .$$
        \item  Assume     $n-3-\mu_2 \geq  \lambda_1$   and  $n-3-\mu_1 \geq  \lambda_2$.
      \begin{enumerate}
        \item  If $\lambda\neq \mu$,  then we further assume $\lambda_1+\lambda_2>\mu_1+\mu_2$ without loss of generality.
         Thus $2(n-3)-(\mu_2+\mu_1)\geq \lambda_1+\lambda_2>\mu_1+\mu_2$.  It follows that $\mu_1+\mu_2<n-3$, and hence $N_{\lambda,\mu}^{\nu,1}=0$ by Theorem \ref{thmQLRforGr3n} (3)(a).
       \item  If $\lambda= \mu$,  then   $n-3-\mu_2=n-3-\lambda_2 \geq  \lambda_1$. If $\lambda_1+\lambda_2<n-3$, then  $N_{\lambda,\mu}^{\nu,1}=0$ by Theorem \ref{thmQLRforGr3n} (3)(a).
         If $\lambda_1+\lambda_2=n-3$, then by the conclusion of part (a), we have $N_{\lambda,\mu}^{\nu,1}=0$ unless $\lambda=\nu$.

         \end{enumerate}
            \end{enumerate}

         It remains to discuss the case $\lambda=\mu=\nu$ with $\lambda_1+\lambda_2=n-3$. We have

   $N_{\lambda,\lambda}^{\lambda,1}=N_{(\lambda_1,\lambda_2,0),(n-3,n-3-\lambda_2,n-3-\lambda_1)}^{(n-3, n-3-\lambda_2,n-3-\lambda_1),1}=  N_{(\lambda_1,\lambda_2,0),(\lambda_1,\lambda_1-\lambda_2, 0)}^{(\lambda_1,\lambda_1-\lambda_2, 0), 1}$.

   If $\lambda_1-\lambda_2<\lambda_2$, then by Theorem \ref{thmdegonerestated} we have
   $$N_{\lambda,\lambda}^{\lambda,1}
   =N_{(n-3-\lambda_2, \lambda_1-\lambda_2,0), (\lambda_1, \lambda_1-\lambda_2, 0)}^{(n-3-\lambda_1+1, n-3-\lambda_2+1, \lambda_1-\lambda_2+1), 0}.$$

   If $\lambda_1-\lambda_2>\lambda_2$, then $N_{\lambda,\lambda}^{\lambda,1}=0$ by the conclusion of part (a).

   If $\lambda_1-\lambda_2=\lambda_2$, then $3\lambda_2=\lambda_1+\lambda_2= n-3$ and $N_{\lambda,\lambda}^{\lambda,1}= -\lambda_2$ by
    Example \ref{exampleccc}.
     \end{proof}

\bibliographystyle{amsplain}

\begin{thebibliography}{99}

\bibitem{AgWo}S. Agnihotri and C. Woodward,\,{\it  Eigenvalues of products of unitary matrices and quantum Schubert calculus}, Math. Res. Lett. 5 (1998), no. 6, 817--836.
  \bibitem{ACT}D. Anderson,  L. Chen  and H.-H. Tseng,\,{\it  On the Finiteness of Quantum $K$-Theory of a Homogeneous Space}, Int. Math. Res. Not. IMRN  2022, no. 2, 1313--1349.
\bibitem{Belk} P. Belkale,\,{\it  Transformation formulas in quantum cohomology}, Compos. Math. 140 (2004), no. 3, 778--792.
\bibitem{Bert} A. Bertram,\,{\it Quantum Schubert calculus}, Adv. Math. 128 (1997), 289--305.
\bibitem{BiCo}S. Billey and I. Coskun,\,{\it Singularities of generalized Richardson varieties}, Comm. Algebra 40 (2012), no. 4, 1466--1495.

\bibitem{Brio} M. Brion,\,{\it Positivity in the Grothendieck group of complex flag varieties},
Special issue in celebration of Claudio Procesi's 60th birthday.
J. Algebra 258 (2002), no. 1, 137--159.
  \bibitem{Buch} A.S. Buch,\, {\it A Littlewood-Richardson rule for the K-theory of Grassmannians}, Acta Math. 189 (2002), no. 1, 37--78.
   \bibitem{BuchGr}A.S. Buch,\,{\it Quantum cohomology of Grassmannians}, Compositio Math. 137 (2003), no. 2, 227--235.
  \bibitem{BCMPfinite00}  A.S. Buch, P.-E. Chaput, L.C.  Mihalcea and N. Perrin, \,{\it Finiteness of cominuscule quantum K-theory}, Ann. Sci. \'Ec. Norm. Sup\'er. (4) 46 (2013), no. 3, 477--494 (2013).
 \bibitem{BCMPfinite} A.S. Buch, P.-E. Chaput, L.C.  Mihalcea and N. Perrin, \,{\it Rational connectedness implies finiteness of quantum K-theory}, Asian J. Math. 20 (2016), no. 1, 117--122.
    \bibitem{BCMP11} A.S. Buch, P.-E. Chaput, L.C.  Mihalcea and N. Perrin, \,{\it Projected Gromov-Witten varieties in cominuscule spaces}, Proc. Amer. Math. Soc. 146 (2018), no. 9, 3647--3660.
      \bibitem{BCMP22}A.S. Buch, P.-E. Chaput, L.C.  Mihalcea and N. Perrin, \,{\it Positivity of minuscule quantum $K$-theory}, preprint available at
         arXiv: math.AG/2205.08630. 
   \bibitem{BCP23} A.S. Buch, P.-E. Chaput and N. Perrin, \,{\it Seidel and Pieri products in cominuscule quantum K-theory}, preprint at arXiv: math.AG/2308.05307.
 \bibitem{BuMi00}A.S. Buch and L.C. Mihalcea,\,{\it Curve neighborhoods of Schubert varieties}, J. Differential Geom. 99 (2015), no. 2, 255--283.
\bibitem{BuMi} A.S. Buch and L.C. Mihalcea,\,{\it  Quantum K-theory of Grassmannians}, Duke Math. J. 156 (2011), no. 3, 501--538.

\bibitem{BuWa}A.S. Buch and C. Wang,\,{\it Positivity determines the quantum cohomology of Grassmannians}, Algebra Number Theory 15 (2021), no. 6, 1505--1521.
 \bibitem{CLLT} K. Chan, S.-C. Lau,  N.C. Leung and H.-H. Tseng,\,{\it Open Gromov-Witten invariants, mirror maps, and Seidel representations for toric manifolds}, Duke Math. J. 166 (2017), no. 8, 1405--1462.
 \bibitem{CMP}P.-E. Chaput, L. Manivel and N. Perrin,\,{\it  Quantum cohomology of minuscule homogeneous spaces. II. Hidden symmetries}, Int. Math. Res. Not. IMRN 2007, no. 22, Art. ID rnm107, 29 pp.
  \bibitem{Curt} C. Curtis,\,{\it A further refinement of the Bruhat decomposition},
Proc. Amer. Math. Soc. 102 (1988), no. 1, 37--42.
 \bibitem{Deba} O. Debarre,\,{\it Higher-Dimensional Algebraic Geometry}, Universitext. Springer, New York (2001).

    \bibitem{FuWo}W.  Fulton and C. Woodward,\,{\it  On the quantum product of Schubert classes}, J. Algebraic Geom. 13 (2004), no. 4, 641--661.
    \bibitem{Give}A.B. Givental,\,{\it On the WDVV equation in quantum $K$-theory}, dedicated to William Fulton on
the occasion of his 60th birthday, Michigan Math. J. 48 (2000), 295--304.
\bibitem{GHS}T. Graber, J. Harris and J. Starr,\,{\it Families of rationally connected varieties}, J. Amer.
Math. Soc. 16 (2003), no. 1, 57--67.
\bibitem{GoIr}E. Gonz\'alez and H. Iritani,\,{\it Seidel elements and mirror transformations}, Selecta Math. (N.S.) 18 (2012), no. 3, 557--590.
 \bibitem{KLS} A. Knutson, T. Lam and D. Speyer,\,{\it Projections of Richardson Varieties},  J. Reine Angew. Math. 687 (2014), 133--157.
 \bibitem{Koll}J. Koll\'ar,\,{\it Higher direct images of dualizing sheaves, I}, Ann. of Math. (2) 123
(1986), 11--42.
\bibitem{Lena}C. Lenart,\,{\it Combinatorial aspects of the $K$-theory of Grassmannians}, Ann. Comb. 4 (2000), no. 1, 67--82.
\bibitem{LiMi} C. Li and L.C.  Mihalcea,\,{\it $K$-theoretic Gromov-Witten invariants of lines in homogeneous spaces}, Int. Math. Res. Not. IMRN 2014, no. 17, 4625--4664.
\bibitem{LiSo} C. Li and J. Song,\,{\it Toward the quantum Pieri rule for $F\ell_n$ via Seidel representation} (in Chinese), SCIENTIA SINICA Mathematica, 2024, 54(12): 2009-2022.
\bibitem{McDu} D. McDuff,\,{\it Quantum homology of fibrations over $S^2$}, Internat. J. Math. 11 (2000), no. 5, 665--721.
\bibitem{McTo} D. McDuff and S. Tolman,\,{\it Topological properties of Hamiltonian circle actions}, IMRP Int. Math. Res.
Pap. (2006), 72826, 1--77.

\bibitem{Post2001}A. Postnikov,\,{\it Symmetries of Gromov-Witten invariants}, Advances in algebraic geometry motivated by physics (Lowell, MA, 2000), 251--258, Contemp. Math., 276, Amer. Math. Soc., Providence, RI, 2001.
\bibitem{Post}A. Postnikov,\,{\it   Affine approach to quantum Schubert calculus}, Duke Math. J. 128 (2005), no. 3, 473--509.
\bibitem{Seid}P. Seidel,\,{\it $\pi_1$ of symplectic automorphism groups and invertibles in quantum homology rings},
Geom. Funct. Anal. 7 (1997), no. 6, 1046--1095.
\bibitem{SSV}B. Shapiro, M. Shapiro and A.  Vainshtein,\,{\it  On combinatorics and topology of pairwise intersections of Schubert cells in $SL_n/{B}$}, The Arnold-Gelfand mathematical seminars, 397--437, Birkh\"auser Boston, Boston, MA, 1997.
\bibitem{ShWi}R.M. Shifler and C. Withrow\,{\it Minimum quantum degrees for isotropic Grassmannians in types $B$ and $C$}, Ann. Comb.   26 (2022), no. 2, 453-480.
 \bibitem{Stri} E.  Strickland,\,{\it Lines in $G/P$}, Math. Z. 242 (2002), no. 2, 227--240.
 \bibitem{Tari} M. Tarigradschi, \,{\it Curve neighborhoods of Seidel products in quantum cohomology}, preprint at arXiv: math.AG/2309.05985.
\end{thebibliography}

\end{document}